\documentclass{article}

\usepackage{amsmath}
\usepackage{amsthm}
\usepackage{natbib}
\usepackage{amssymb}
\usepackage{times}
\usepackage{verbatim}
\usepackage{natbib}

\newtheorem{proposition}{Proposition}
\newtheorem{theorem}{Theorem}
\newtheorem{corollary}{Corollary}
\newtheorem{lemma}{Lemma}
\newtheorem{assumption}{Assumption}

\newcommand{\vanden}{N}
\newcommand{\hatvap}{\hat p}
\DeclareMathOperator{\Var}{Var}
\DeclareMathOperator{\E}{E}

\newcommand{\nden}{n}
\newcommand{\nnum}{r}
\newcommand{\fdos}{\mu_2}
\newcommand{\funo}{\mu_1}
\newcommand{\fest}{g}
\newcommand{\Fe}{\mathcal F}
\newcommand{\Felp}{\mathcal F_{\mathrm p}}
\newcommand{\Feln}{\mathcal F_{\mathrm n}}
\newcommand{\FeA}{\mathcal F_{\mathcal R}}
\newcommand{\festa}{G}
\newcommand{\st}{\mid} 
\newcommand{\ff}[2]{#1^{(#2)}} 
\newcommand{\diff}{\mathrm d}
\newcommand{\risk}{\eta}
\newcommand{\riskas}{\bar \eta}
\newcommand{\Minf}{M'_\mathrm{L}}
\newcommand{\Mcero}{M_\mathrm{L}}
\newcommand{\minf}{m'_\mathrm{L}}
\newcommand{\mcero}{m_\mathrm{L}}
\newcommand{\Kinf}{{K'}}
\newcommand{\Kcero}{{K}}
\newcommand{\xinf}{{x'_\mathrm L}}
\newcommand{\xcero}{{x_\mathrm L}}
\newcommand{\xinfcr}{\xi'}
\newcommand{\xcerocr}{\xi}
\newcommand{\riskasinf}{\zeta'}
\newcommand{\riskascero}{\zeta}
\newcommand{\riskasfs}{\zeta}
\newcommand{\numenos}{\alpha}
\newcommand{\numas}{\beta}
\newcommand{\ndenini}{u}
\newcommand{\ndenfin}{v}
\newcommand{\csopfin}{\sigma}
\newcommand{\epsf}{\mathrm{unif}}
\newcommand{\epsg}{\mathrm{est}}
\newcommand{\epsc}{\mathrm{cont}}
\newcommand{\epsd}{\mathrm{disc}} 
\newcommand{\epsi}{\mathrm{int}}
\newcommand{\epst}{\mathrm{tail}}
\newcommand{\epsa}{\mathrm{interv}}
\newcommand{\epsh}{\mathrm{const}}
\newcommand{\epscuant}{9}
\newcommand{\psiL}{T}
\newcommand{\ndenpk}{\nu_{\nden,k}}

\makeatletter
\renewcommand\theenumii{(\@alph\c@enumii)}

\renewcommand\p@enumii{}
\makeatother

\begin{document}

\title{Asymptotically Optimum Estimation of a Probability in Inverse Binomial Sampling\\ under General Loss Functions}

\author{Luis Mendo \thanks{E.T.S. Ingenieros de Telecomunicaci\'on, Polytechnic University of Madrid, 28040 Madrid, Spain. E-mail: lmendo@grc.ssr.upm.es.}}
\date{\today}

\maketitle

\begin{abstract}
The optimum quality that can be asymptotically achieved in the estimation of a probability $p$ using inverse binomial sampling is addressed.
A general definition of quality is used in terms of the risk associated with a loss function that satisfies certain assumptions. It is shown that the limit superior of the risk for $p$ asymptotically small has a minimum over all (possibly randomized) estimators. This minimum is achieved by certain non-randomized estimators. The model includes commonly used quality criteria as particular cases. Applications to the non-asymptotic regime are discussed considering specific loss functions, for which minimax estimators are derived.

\end{abstract}

\section{Introduction}
\label{parte: intro}

The problem of sequentially estimating the probability of success, $p$, in a sequence of Bernoulli trials arises in many fields of science and engineering. A stopping rule of notable interest, first discussed by \citet{Haldane45}, is \emph{inverse binomial sampling}, which consists in observing the random sequence until a given number $\nnum$ of successes are obtained. The resulting number of trials, $\vanden$, is a sufficient statistic \citep[p.~101]{Lehmann98}, from which $p$ can be estimated. The appeal of this rule lies in the useful properties of estimators obtained from it. Namely, previous works have shown that the uniformly minimum variance unbiased estimator, given by \citep{Haldane45}
\begin{equation}
\label{eq: hatvap nnum 1 vanden 1}
\hatvap = \frac{\nnum-1}{\vanden-1},
\end{equation}
satisfies the following properties. Its normalized mean square error $\E[(\hatvap-p)^2]/p^2$ has an asymptotic value for $\nnum \geq 3$, namely $1/(\nnum-2)$; and $\E[(\hatvap-p)^2]/p^2$ is guaranteed to be smaller than this value for any $p \in (0,1)$ \citep{Mikulski76}. Similarly, the normalized mean absolute error $\E[|p-\hatvap|]/p$ is smaller than its asymptotic value, given by $2 (\nnum-1)^{\nnum-2} \exp(-\nnum+1) / (\nnum-2)!$, for any $p \in (0,1)$ and $\nnum \geq 2$ \citep{Mendo09a}. In addition, given $\funo, \fdos > 1$ and $\nnum \geq 3$, under certain conditions this estimator, as well as the modified version $\hatvap = (\nnum-1)/\vanden$, can guarantee that, for $p$ arbitrary, the random interval $[\hatvap/\funo, \hatvap\fdos]$ contains the true value $p$ with a confidence level greater than a prescribed value \citep{Mendo06,Mendo08a}.

The results mentioned apply to specific estimators, defined as functions of the sufficient statistic $\vanden$. A natural extension is to investigate whether the quality of the estimation can be improved using other estimators. The most general class is that formed by randomized estimators defined in terms of $\vanden$. This includes non-randomized estimators as a particular case. This problem is addressed by \citet{Mendo10}, using the confidence associated with a relative interval as a quality measure. It is shown that the confidence that can be guaranteed for $p$ asymptotically small has a maximum over all estimators. Moreover, non-randomized estimators are given that can guarantee this maximum confidence not only asymptotically, but also for $p \in (0,1)$ arbitrary.

A further generalization is to consider arbitrary estimators with an arbitrary definition of quality. The present paper pursues this direction, focusing on the asymptotic regime. Namely, quality is defined as the risk associated with an arbitrary loss function. The allowed loss functions are restricted only by certain regularity conditions, which are easily satisfied in practice (and which, in particular, hold for all the previously mentioned examples of quality measures). Using this general definition of quality, the asymptotic performance as $p \rightarrow 0$ of arbitrary estimators in inverse binomial sampling is analyzed. As will be seen, the quality that can be asymptotically achieved has a maximum over all estimators. Furthermore, this maximum can be accomplished using certain non-randomized estimators, whose form is explicitly given.

Section~\ref{parte: prel} contains preliminary definitions and observations required for the main results, which are presented in Section~\ref{parte: res}. Section~\ref{parte: disc} discusses these results, and considers applications in the non-asymptotic regime.
Proofs of all results are given in \ref{parte: proofs}.

\section{Preliminaries}
\label{parte: prel}

The following notation will be used. Let $\ff{k}{i}$ denote $k(k-1) \cdots (k-i+1)$, for $k \in \mathbb Z$, $i \in \mathbb N$; and $\ff{k}{0}=1$. Given $\nnum \in \mathbb N$, the probability function of $\vanden$, $f(\nden) = \Pr[\vanden = \nden]$, is
\begin{equation}
\label{eq: f}
f(\nden)
= \frac{\ff{(\nden-1)}{\nnum-1}}{(\nnum-1)!} p^\nnum (1-p)^{\nden-\nnum}, \quad \nden \geq \nnum.
\end{equation}
The upper and lower (not normalized) incomplete gamma functions are respectively denoted as
\begin{align}
\label{eq: Gamma}
\Gamma(s,u) & = \int_u^\infty \tau^{s-1}\exp(-\tau) \,\diff \tau, \\
\label{eq: gamma}
\gamma(s,u) & = \int_0^u \tau^{s-1}\exp(-\tau) \,\diff \tau = \Gamma(s) - \Gamma(s,u).
\end{align}
In addition, the functions $\phi(\nu)$ and $\psi(x,\Omega)$ are defined as
\begin{align}
\label{eq: phi}
\phi(\nu) & = \frac{\nu^{\nnum-1} \exp(-\nu)}{(\nnum-1)!}, \quad \nu \in \mathbb R^+, \\
\label{eq: psi}
\psi(x,\Omega) & = \frac{\Omega^\nnum \exp(-\Omega/x)}{x^{\nnum+1} (\nnum-1)!}, \quad x, \Omega \in \mathbb R^+.
\end{align}

Given a function $h$, the one-sided limits $\lim_{x \rightarrow a^-}h(x)$ and $\lim_{x \rightarrow a^+}h(x)$ are respectively denoted as $h(a-)$ and $h(a+)$. Given two functions $h_1, h_2: \mathbb R^+ \mapsto \mathbb R^+ \cup \{0\}$, $h_1(x)$ is $O(h_2(x))$ as $x \rightarrow \infty$ (respectively as $x \rightarrow 0$) if and only if there exist $a, M \in \mathbb R^+$ such that $h_1(x) \leq M h_2(x)$ for all $x \geq a$ (respectively for all $x \leq a$). Similarly, $h_1(x)$ is $\Theta(h_2(x))$ as $x \rightarrow \infty$ (respectively as $x \rightarrow 0$) if and only if there exist $a, m, M \in \mathbb R^+$ such that $m h_2(x) \leq h_1(x) \leq M h_2(x)$ for all $x \geq a$ (respectively for all $x \leq a$).

The quality of an estimator $\hatvap$ is measured by the \emph{risk} (expected loss) $\risk = \E[L(\hatvap/p)]$ associated with a non-negative \emph{loss function} $L: \mathbb R^+ \mapsto \mathbb R^+ \cup \{0\}$, provided that this expectation exists. The function $L$ is defined in terms of $\hatvap/p$, rather than $\hatvap$. This is motivated by the fact that a given error value is most meaningful when compared with $p$, and therefore commonly used quality measures are most often \emph{normalized} ones.

The loss function is assumed to satisfy the following.

\begin{assumption}
\label{assum: L var acot}
For any $x_1,x_2 \in \mathbb R^+$ with $x_2>x_1$, $L$ is of bounded variation on $[x_1,x_2]$.
\end{assumption}

\begin{assumption}
\label{assum: L disc fin compacto}
For any $x_1,x_2 \in \mathbb R^+$ with $x_2>x_1$, $L$ has a finite number of discontinuities in $[x_1,x_2]$.
\end{assumption}

\begin{assumption}
\label{assum: L cota O}
The loss function has the following asymptotic behaviour:
\begin{enumerate}
\item
\label{assum item: cota O, cero}
There exists $\Kcero \in \mathbb R$ such that $L(x)$ is $O(x^\Kcero)$ as $x \rightarrow 0$.
\item
\label{assum item: cota O, inf}
There exists $\Kinf<\nnum$ such that $L(x)$ is $O(x^\Kinf)$ as $x \rightarrow \infty$.
\end{enumerate}
\end{assumption}

These restrictions are very mild. Note that the loss function $L$ is not required to be convex, or continuous; however, being of bounded variation implies that its discontinuities can only be jumps or removable discontinuities, i.e.~$L$ has left-hand and right-hand limits at every point of its domain, and these limits are finite \citep[corollary 2.7.3]{Carter00}. All quality measures mentioned in Section~\ref{parte: intro} can be expressed in terms of functions of $x = \hatvap/p$ for which Assumptions~\ref{assum: L var acot}--\ref{assum: L cota O}
hold. Namely, $L(x) = (x-1)^2$ corresponds to normalized mean square error; $L(x) = |x-1|$ to normalized mean absolute error; and given $\funo, \fdos > 1$,
\begin{equation}
\label{eq: L conf}
L(x) =
\begin{cases}
0 & \text{if } x \in [1/\fdos, \funo], \\
1 & \text{otherwise}
\end{cases}
\end{equation}
corresponds to $1$ minus the confidence associated with a relative interval $[p/\fdos, p\funo]$.

Since $\vanden$ is a sufficient statistic, for any estimator defined in terms of the observed sequence of Bernoulli variables for which $\E[L(\hatvap/p)]$ exists, there is a possibly randomized estimator expressed only in terms of $\vanden$ that has the same risk \citep[p.~33]{Lehmann98}. Therefore, attention can be restricted to estimators that depend on the observations through $\vanden$ only; however, randomized estimators need to be considered in addition to non-randomized ones.

The set of all functions from $\{\nnum, \nnum+1, \nnum+2,\ldots\}$ to $\mathbb R^+$ is denoted as $\Fe$. A \emph{non-randomized estimator} $\hatvap$ is defined as $\hatvap = \fest(\vanden)$, with $\fest \in \Fe$. A \emph{randomized estimator} is a positive random variable $\hatvap$ whose distribution depends on the value of $\vanden$. The distribution function of $\hatvap$ conditioned on $\vanden = \nden$ will be denoted as $\Pi_\nden$. The randomized estimator is completely specified by the functions $\Pi_\nden$, $\nden \geq \nnum$. Denoting by $\FeA$ the class of all functions from $\{\nnum, \nnum+1, \nnum+2,\ldots\}$ to the set of distribution functions, a randomized estimator is defined by a function $\festa \in \FeA$ that to each $\nden$ assigns $\Pi_\nden$. Clearly, non-randomized estimators form a subset of the class of randomized estimators. Throughout the paper, when referring to an arbitrary estimator without specifying its type, the general class of randomized estimators (including non-randomized ones) will be meant.

The risk will be explicitly denoted in the sequel as a function of $p$, that is, $\risk(p)$. For a non-randomized estimator defined by $\fest \in \Fe$, the risk $\risk(p)$ is given by
\begin{equation}
\label{eq: risk non-rand}
\risk(p) = \sum_{\nden=\nnum}^\infty f(\nden) L(\fest(\nden)/p).
\end{equation}
Depending on $L$, $\fest$ and $p$, this series may be convergent or not; however, boundedness of $\fest$ is sufficient to ensure that the series converges for all $L$ satisfying Assumptions~\ref{assum: L var acot}--\ref{assum: L cota O}
and for all $p$. In general, for possibly randomized estimators,
\begin{equation}
\label{eq: risk gen}
\risk(p) = \sum_{\nden=\nnum}^\infty f(\nden) \int_0^\infty L (y/p) \,\diff \Pi_\nden(y),
\end{equation}
where the integral is defined in the Lebesgue-Stieltjes sense. Assumptions~\ref{assum: L var acot}--\ref{assum: L cota O}
assure that this integral always exists; however, it may be finite or infinite. Besides, even if it is finite for a given $p$ and for all $\nden$, the series in \eqref{eq: risk gen} does not necessarily converge for that $p$. According to this, for an arbitrary estimator and for $p$ given, $\risk(p)$ may be finite or infinite; however, there exist estimators that have a finite risk for all $p$.

An arbitrary estimator may not have an asymptotic risk, i.e.~$\lim_{p \rightarrow 0} \risk(p)$ need not exist in general. Therefore, the asymptotic behaviour of an estimator should be characterized by $\limsup_{p \rightarrow 0} \risk(p)$. The significance of the limit superior lies in the fact that it is the smallest value such that any greater number is asymptotically an upper bound of $\risk(p)$. That is, given any $\risk_0>\limsup_{p \rightarrow 0} \risk(p)$, there exists $\delta>0$ such that $\risk(p) < \risk_0$ for all $p < \delta$; and no such $\delta$ can be found for $\risk_0<\limsup_{p \rightarrow 0}  \risk(p)$.\footnote{For $\risk_0=\limsup_{p \rightarrow 0} \risk(p)$ the result may hold or not depending on the estimator and loss function; for example, it holds for \eqref{eq: hatvap nnum 1 vanden 1} and normalized mean square error, as mentioned in Section~\ref{parte: intro}, whereas it obviously does not hold for a constant loss function.}

According to the preceding discussion, a desirable asymptotic property of an estimator is that it achieves a low value of $\limsup_{p\rightarrow 0} \risk(p)$. In order to characterize how low this value can be, the infimum of $\limsup_{p\rightarrow 0} \risk(p)$ over all estimators should be determined. A related question is whether there is an estimator that can attain this infimum. As will be seen, the answer to this question is affirmative, that is, the infimum is also a minimum. This implies that there exist optimum estimators from the point of view of asymptotic behaviour; moreover, they can be found within the class of non-randomized estimators, as will also be shown. To obtain these results, the following approach will be used. It will be first established that for a certain subclass of non-randomized estimators, $\lim_{p \rightarrow 0} \risk(p)$ exists and can be easily computed. Secondly, it will be proved that $\lim_{p \rightarrow 0} \risk(p)$ has a minimum  value over the referred subclass. Thirdly, this minimum will be shown to coincide with the unrestricted minimum of $\limsup_{p\rightarrow 0} \risk(p)$ over the class of arbitrary estimators.

\section{Main results}
\label{parte: res}

For a given loss function $L$, the set of all functions $\fest \in \Fe$ such that $\lim_{p \rightarrow 0} \risk(p)$ exists for $\hatvap = \fest(\nden)$ is denoted as $\Felp$. The set of functions $\fest \in \Fe$ for which $\lim_{\nden \rightarrow \infty} \nden \fest(\nden)$ exists, is finite and non-zero is denoted as $\Feln$. Observe that the definition of $\Felp$ generalizes that given by \citet{Mendo10}, which assumes a specific loss function, namely \eqref{eq: L conf}. The result in Theorem~\ref{teo: Feln incl Felp} to follow establishes that $\Feln \subseteq \Felp$, and explicitly gives $\lim_{p \rightarrow 0} \risk(p)$. For any $\fest \in \Feln$ with $\lim_{\nden \rightarrow \infty} \nden \fest(\nden) = \Omega$, let
\begin{equation}
\label{eq: riskas nu}
\riskas = \int_0^\infty \phi(\nu) L( \Omega/\nu) \,\diff \nu.
\end{equation}
Equivalently, $\riskas$ can be expressed as
\begin{equation}
\label{eq: riskas x}
\riskas = \int_0^\infty \psi(x,\Omega) L(x) \,\diff x
\end{equation}
by means of the change of variable $\nu = \Omega/x$ (both expressions are used in the proofs of the results to be presented). By Assumptions~\ref{assum: L var acot} and \ref{assum: L cota O}, these integrals exist as improper Riemann integrals, and have a finite value. It should be observed (and is exploited in the proofs) that they can also be interpreted as Lebesgue integrals \citep[theorem 10.33]{Apostol74}.

\begin{theorem}
\label{teo: Feln incl Felp}
Consider $\nnum \in \mathbb N$.
For any loss function satisfying Assumptions~\ref{assum: L var acot}--\ref{assum: L cota O},
and for any non-randomized estimator defined by a function $\fest \in \Feln$, the limit $\lim_{p \rightarrow 0} \risk(p)$ exists and equals $\riskas$ given by \eqref{eq: riskas nu} (or \eqref{eq: riskas x}).
\end{theorem}

According to this, the asymptotic risk of an estimator defined by any function $\fest \in \Feln$ depends on this function only through $\Omega$, i.e.~only the asymptotic behaviour of $\fest$ matters. Furthermore, under an additional assumption, it can be shown that the asymptotic risk is a $C^1$ function of $\Omega$.

\setcounter{assumption}{1} 
{ 
\makeatletter
\renewcommand\theassumption{\@arabic\c@assumption'}
\makeatother
\begin{assumption}
\label{assum: L disc fin}
$L$ has a finite number of discontinuities in $\mathbb R^+$.
\end{assumption}
} 

It is evident that
Assumption~\ref{assum: L disc fin} implies
Assumption~\ref{assum: L disc fin compacto}. While more restrictive,
Assumption~\ref{assum: L disc fin} is satisfied by a large class of loss functions, including the mentioned examples.

\begin{proposition}
\label{prop: riskas C1 Omega}
Given $\nnum \in \mathbb N$, a loss function satisfying Assumptions~\ref{assum: L var acot}, \ref{assum: L disc fin} and \ref{assum: L cota O}, and an estimator defined by a function $\fest \in \Feln$, the asymptotic risk $\riskas$ is a $C^1$ function of $\Omega \in \mathbb R^+$, with
\begin{equation}
\label{eq: der riskas}
\frac{\diff \riskas}{\diff \Omega} =
\int_0^\infty \frac{\partial \psi(x,\Omega)}{\partial \Omega} L(x) \,\diff x.
\end{equation}
Denoting by ${\riskas|}_\nnum$ the asymptotic risk corresponding to $\Omega$ and $\nnum$ given, this derivative can be expressed as
\begin{equation}
\label{eq: der riskas Omega}
\frac{\diff {\riskas|}_\nnum}{\diff \Omega} = \frac {\nnum ({\riskas|}_\nnum - {\riskas|}_{\nnum+1})} \Omega.
\end{equation}
\end{proposition}

Within the restricted class of non-randomized estimators defined by $\Feln$, it is natural to search for values of $\Omega$ that yield low values of the asymptotic risk $\riskas$. Depending on the loss function, there may be or not an optimum value of $\Omega \in \mathbb R^+$, in the sense of minimizing $\riskas$. Theorem~\ref{teo: min resp Omega} to follow establishes that, under certain additional hypotheses (represented by Assumption~\ref{assum: L crec}), $\riskas$ indeed has a minimum with respect to $\Omega$.

\setcounter{assumption}{3} 
\begin{assumption}
\label{assum: L crec}
The loss function satisfies the following properties:
\begin{enumerate}
\item
\label{assum item: bien izq}
There exists $\xcerocr \in \mathbb R^+$ such that $L$ is non-increasing on $(0,\xcerocr)$ and
\begin{equation}
\label{eq: cond 0, integral}
\int_{\xcerocr}^\infty \frac{L(\xcerocr)-L(x)}{x^{\nnum+1}} \,\diff x > 0.
\end{equation}
\item
\label{assum item: bien der}
There exists $\xinfcr \in \mathbb R^+$ such that $L$ is non-decreasing on $(\xinfcr,\infty)$ and one of these conditions holds:
\begin{enumerate}
\item
\label{assum item item: i}
$L(\xinfcr-) < L(\xinfcr+)$.
\item
\label{assum item item: ii}
There is $t \in \mathbb N$ such that $L$ is of class $C^t$ on an interval containing $\xinfcr$ and
\begin{align}
\left. \frac{\diff^i L}{\diff x^i} \right|_{x=\xinfcr} & = 0 \quad \text{for } i=1,2,\ldots,t-1, \\
(-1)^{t-1} \left. \frac{\diff^t L}{\diff x^t} \right|_{x=\xinfcr} & > 0.
\end{align}
\end{enumerate}
\end{enumerate}
\end{assumption}

The next proposition gives a sufficient condition that may help in assessing whether a given loss function satisfies property~\ref{assum item: bien izq} in Assumption~\ref{assum: L crec}.

\begin{proposition}
\label{prop: assum item bien izq, lim}
If there exist $A \in \mathbb R$ and $B, s$ such that
\begin{equation}
\label{eq: cond 0, lim}
\lim_{x \rightarrow 0} \frac{L(x) - A}{x^s} = B \quad \text{with } Bs < 0,\ s < \nnum,
\end{equation}
inequality \eqref{eq: cond 0, integral} holds for some $\xcerocr \in \mathbb R^+$.
\end{proposition}

\begin{theorem}
\label{teo: min resp Omega}
Given $\nnum \in \mathbb N$ and a loss function satisfying Assumptions~\ref{assum: L var acot}, \ref{assum: L disc fin}, \ref{assum: L cota O} and \ref{assum: L crec}, consider the class of non-randomized estimators defined by functions $\fest \in \Feln$. Denoting $\Omega = \lim_{\nden \rightarrow \infty} \nden \fest(\nden)$, there exists a value of $\Omega$ which minimizes the asymptotic risk $\riskas$ among all $\Omega \in \mathbb R^+$.
\end{theorem}

This theorem indicates that in the stated conditions, and restricted to the class defined by $\Feln$, there is an optimum value of $\Omega$ from the point of view of asymptotic risk. This optimum is not necessarily unique. In the sequel, $\risk^*$ will denote the minimum of $\riskas$ over the class of estimators defined by $\Feln$, and $\Omega^*$ will denote any value of $\Omega$ which attains this minimum, that is,
\begin{equation}
\label{eq: riskas opt nu}
\risk^* = \int_0^\infty \phi(\nu) L( \Omega^*/\nu) \,\diff \nu.
\end{equation}

Assumption~\ref{assum: L crec} holds for a wide range of loss functions, and in particular for those corresponding to normalized mean square error, normalized mean absolute error, and confidence associated with a relative interval. It is not difficult, however, to find a loss function for which the assumption does not hold, and for which $\riskas$ does not have a minimum over the class defined by $\Feln$. For example, given $A_1, A_2 >0$, let
\begin{equation}
\label{eq: L conf gen}
L(x) =
\begin{cases}
0 & \text{if } x \in [1/\fdos, \funo], \\
A_2 & \text{if } x < 1/\fdos, \\
A_1 & \text{if } x > \funo,
\end{cases}
\end{equation}
which is a generalized version of \eqref{eq: L conf}. Substituting \eqref{eq: L conf gen} into \eqref{eq: cond 0, integral}, it is seen that property~\ref{assum item: bien izq} in Assumption~\ref{assum: L crec} is satisfied if and only if
\begin{equation}
\label{eq: cond L conf gen}
\frac{A_1}{A_2} < (\funo\fdos)^\nnum,
\end{equation}
while property~\ref{assum item: bien der} holds irrespective of $A_1$ and $A_2$. On the other hand, for $\Omega \in \mathbb R$, substituting \eqref{eq: L conf gen} into \eqref{eq: riskas nu} and computing $\diff\riskas/\diff\Omega$ gives
\begin{equation}
\frac{\diff\riskas}{\diff\Omega}
= \frac{\Omega^{\nnum-1} \left( A_1 \funo^{-\nnum} \exp(-\Omega/\funo) -A_2 \fdos^\nnum \exp(-\Omega\fdos) \right)}{(\nnum-1)!}.
\end{equation}
This implies that $\riskas$ has a single minimum over $\Omega \in \mathbb R$, located at
\begin{equation}
\Omega = \frac{\nnum \log(\funo\fdos) - \log(A_1/A_2)}{\fdos-1/\funo}.
\end{equation}
This value is positive if and only if \eqref{eq: cond L conf gen}, or equivalently property~\ref{assum item: bien izq} in Assumption~\ref{assum: L crec}, is satisfied. Thus, if this property does not hold, $\riskas$ is monotonically increasing for $\Omega \in \mathbb R^+$, which implies that there is not an optimum $\Omega$ within $\mathbb R^+$.

Under the hypotheses of Theorem~\ref{teo: min resp Omega}, the optimum value of $\Omega$ for the considered $\nnum$, i.e.~$\Omega^*$, satisfies, by Proposition~\ref{prop: riskas C1 Omega},
\begin{equation}
\label{eq: der 0}
\frac{\diff\riskas}{\diff\Omega} = 0
\end{equation}
(or equivalently, using the notation in the referred proposition, $\riskas|_{\nnum} = \riskas|_{\nnum+1}$). Thus if \eqref{eq: der 0} has only one solution, it must be equal to $\Omega^*$. If there are several solutions, at least one corresponds to the absolute minimum of $\riskas$, although not necessarily all of them do.

According to Theorem~\ref{teo: min resp Omega}, if the loss function satisfies Assumptions~\ref{assum: L var acot}, \ref{assum: L disc fin}, \ref{assum: L cota O} and \ref{assum: L crec}, any non-randomized estimator defined by a function $\fest \in \Feln$ with $\lim_{\nden \rightarrow \infty} \nden \fest(\nden) = \Omega^*$ minimizes $\limsup_{p \rightarrow 0} \risk(p)$ within the restricted class of estimators represented by $\Feln$; but not necessarily within the class of all non-randomized estimators, or within the general class of possibly randomized estimators. However, under slightly stronger conditions this turns out to be true, as established by the next theorem.

\setcounter{assumption}{2} 
{ 
\makeatletter
\renewcommand\theassumption{\@arabic\c@assumption'}
\makeatother
\begin{assumption}
\label{assum: L cota Theta}
The loss function has the following asymptotic behaviour:
\begin{enumerate}
\item
\label{assum item: cota Theta, cero}
There exists $\Kcero<\nnum$ such that $L(x)$ is $\Theta(x^\Kcero)$ as $x \rightarrow 0$.
\item
\label{assum item: cota Theta, inf}
There exists $\Kinf<\nnum$ such that $L(x)$ is $\Theta(x^\Kinf)$ as $x \rightarrow \infty$.
\end{enumerate}
\end{assumption}
} 

Assumption~\ref{assum: L cota Theta} replaces Assumption~\ref{assum: L cota O}, in the sense that
each of the two properties in Assumption~\ref{assum: L cota Theta} implies the corresponding one in Assumption~\ref{assum: L cota O}. The new conditions are only slightly more restrictive, and are still satisfied by a large set of loss functions, in particular by those previously mentioned as examples.

\begin{theorem}
\label{teo: opt}
Given $\nnum \in \mathbb N$ and any loss function satisfying Assumptions~\ref{assum: L var acot}, \ref{assum: L disc fin}, \ref{assum: L cota Theta} and \ref{assum: L crec}, $\limsup_{p \rightarrow 0} \risk(p)$ has a minimum over the general class of estimators defined by $\FeA$, and this minimum equals $\risk^*$.
\end{theorem}

\begin{corollary}
\label{cor: opt, vale Feln}
Under the hypotheses of Theorem~\ref{teo: opt}, any non-randomized estimator defined by a function $\fest \in \Feln$ with $\lim_{\nden \rightarrow \infty} \nden \fest(\nden) = \Omega^*$ minimizes $\limsup_{p \rightarrow 0} \risk(p)$ among all (possibly randomized) estimators based on inverse binomial sampling.
\end{corollary}

Theorem~\ref{teo: opt} and Corollary~\ref{cor: opt, vale Feln} show that, under the stated assumptions, an estimator can be found within the class defined by $\Feln$ that is asymptotically optimum over the general class represented by $\FeA$.

\section{Discussion and applications}
\label{parte: disc}


Since $p$ is unknown, it is desirable to have an estimator that \emph{guarantees} that the risk is not larger than a given $\risk_0$ for $p$ arbitrary, or at least for all $p$ within a certain interval; that is, such that $\risk(p) \leq \risk_0$ for $p$ in some interval $(p_1,p_2)$, with $0 \leq p_1 < p_2 \leq 1$. If $p_1 = 0$, the estimator is  said to \emph{asymptotically guarantee} that the risk is not larger than $\risk_0$; if, in addition, $p_1=1$, it \emph{globally guarantees} that the risk is not larger than $\risk_0$.

The results presented in Section~\ref{parte: res} generalize the asymptotic analysis by \citet{Mendo10}, which considers the specific loss function \eqref{eq: L conf}, to arbitrary functions satisfying the indicated assumptions. The importance of these asymptotic results lies not only in the fact that in many applications $p$ is small, but also in the observation that asymptotic behaviour sets a restriction on the risk that can be guaranteed. This restriction is represented by the following proposition (which is a straightforward generalization of \citet[proposition~1]{Mendo10}) and its corollary.

\begin{proposition}
\label{prop: as nonas}
If an estimator has a risk $\risk(p)$ not larger than a given $\risk_0$ for all $p \in (p_1,p_2)$, then necessarily $\limsup_{p \rightarrow p_0} \risk(p) \leq \risk_0$ for any $p_0 \in [p_1,p_2]$.
\end{proposition}

\begin{corollary}
\label{cor: no existe}
Given $\nnum \in \mathbb N$ and a loss function that satisfies Assumptions~\ref{assum: L var acot}, \ref{assum: L disc fin}, \ref{assum: L cota Theta} and \ref{assum: L crec}, for any $\risk_0 < \risk^*$ and $p_2>0$, no estimator can guarantee that $\risk(p) \leq \risk_0$ for all $p < p_2$.
\end{corollary}

According to the results in Section~\ref{parte: res}, if Assumptions~\ref{assum: L var acot}, \ref{assum: L disc fin}, \ref{assum: L cota Theta} and \ref{assum: L crec} are satisfied, any estimator defined by $\fest \in \Feln$ with $\lim_{\nden \rightarrow \infty} \nden \fest(\nden) = \Omega^*$ can asymptotically guarantee that the risk is not larger than $\risk^* + \epsilon$ for any $\epsilon > 0$, whereas Corollary~\ref{cor: no existe} states that no estimator exists with this property for $\epsilon < 0$. It remains to be seen if there exist estimators that asymptotically guarantee that $\risk(p) \leq \risk^*$; and, particularly, if this guarantee can be global. The answer to these questions depends on the loss function under consideration. Since a general analysis seems impracticable, a separate study needs to be carried out for each loss function. Several important cases are discussed next, including the loss functions already mentioned as examples.

\subsection{Confidence}

For the loss function given by \eqref{eq: L conf}, $\risk(p)$ equals $1-c(p)$, where $c(p) = \Pr[p/\fdos \leq \hatvap \leq p\funo] = \Pr[\hatvap/\funo \leq p \leq \hatvap \fdos]$ is the \emph{confidence} associated with a relative interval defined by $\funo, \fdos>1$. Let $c^* = 1-\risk^*$, which represents the maximum confidence that could be guaranteed to be exceeded. The analysis by \citet{Mendo10} shows that assuming $\nnum \geq 3$, the inequality $c(p) > c^*$ can indeed be asymptotically guaranteed for any $\funo$, $\fdos$, and globally guaranteed if $\funo, \fdos$ satisfy certain conditions.

\subsection{Mean absolute error}

For $L(x) = |x-1|$, risk corresponds to \emph{normalized mean absolute error}. Considering an estimator $\hatvap = \fest(\vanden)$ with $\lim_{\nden \rightarrow \infty} \nden \fest(\nden) = \Omega$, and for $\nnum \geq 2$, \eqref{eq: riskas nu} gives the asymptotic risk
\begin{equation}
\label{eq: riskas MAE}
\riskas = \int_0^\infty \phi(\nu) \left| \frac{\Omega}{\nu} - 1 \right| \,\diff \nu
= \frac{2\left(\Gamma(\nnum,\Omega)-\Omega\Gamma(\nnum-1,\Omega\right))}{(\nnum-1)!} + \frac{\Omega}{\nnum-1} - 1,
\end{equation}
and it is straightforward to show that \eqref{eq: der 0} reduces to $\Gamma(\nnum-1,\Omega)= (\nnum-2)!/2$. This equation has only one solution, which thus corresponds to $\Omega^*$. Interestingly, for $\hatvap = \Omega^*/(\nden-1)$ with $\nnum \geq 2$, numerically evaluating $\risk(p)$ suggests that this estimator may globally guarantee $\risk(p) \leq \risk^*$. However, proving this conjecture remains an open problem.

\subsection{Mean square error}

The function $L(x) = (x-1)^2$ corresponds to \emph{normalized mean square error}. This loss function lends itself easily to non-asymptotic analysis. Considering an estimator $\hatvap = \fest(\vanden)$ with $\lim_{\nden \rightarrow \infty} \nden \fest(\nden) = \Omega$, and assuming $\nnum \geq 3$, \eqref{eq: riskas nu} gives
\begin{equation}
\label{eq: riskas MSE}
\riskas = \int_0^\infty \phi(\nu) \left( \frac{\Omega}{\nu} - 1 \right)^2 \diff \nu
= \frac{\Omega^2}{(\nnum-1)(\nnum-2)} - \frac{2\Omega}{\nnum-1}+1,
\end{equation}
and thus \eqref{eq: der 0} has the single solution $\Omega=\nnum-2$, which is the optimum value for $\Omega$, i.e.~$\Omega^*$. From \eqref{eq: riskas MSE} the resulting $\risk^*$ is $1/(\nnum-1)$. As established by the next proposition, an estimator can be found that globally guarantees that the risk is not larger than $\risk^*$, namely
\begin{equation}
\label{eq: hatvap nnum 2 vanden 1}
\hatvap = \frac{\nnum-2}{\vanden-1}.
\end{equation}

\begin{proposition}
\label{prop: MSE gar}
Given $\nnum \geq 3$, and for any $p \in (0,1)$, the estimator \eqref{eq: hatvap nnum 2 vanden 1} satisfies
\begin{equation}
\label{eq: MSE gar}
\frac{\E[(\hatvap -p)^2]}{p^2} < \frac{1}{\nnum-1}.
\end{equation}
\end{proposition}

The following corollary is obtained from Theorem~\ref{teo: opt} and Proposition~\ref{prop: MSE gar}.

\begin{corollary}
\label{cor: opt global gar NMSE}
For $\nnum \geq 3$, the estimator \eqref{eq: hatvap nnum 2 vanden 1} minimizes $\sup_{p \in (0,1)} \E[(\hatvap -p)^2]/p^2$ among all (possibly randomized) estimators based on inverse binomial sampling.
\end{corollary}

Thus the estimator given by \eqref{eq: hatvap nnum 2 vanden 1} not only minimizes $\limsup_{p \rightarrow 0} \E[(\hatvap -p)^2]/p^2$, but also $\sup_{p \in (0,1)} \E[(\hatvap -p)^2]/p^2$, i.e.~it is minimax with respect to normalized mean square error. Therefore, from the point of view of guaranteeing that the normalized mean square error does not exceed a given value, \eqref{eq: hatvap nnum 2 vanden 1} is optimum among all estimators based on inverse binomial sampling.

Comparing the estimators \eqref{eq: hatvap nnum 1 vanden 1} and \eqref{eq: hatvap nnum 2 vanden 1}, the former can only guarantee $\E[(\hatvap -p)^2]/p^2 < 1/(\nnum-2)$, whereas the latter guarantees $\E[(\hatvap -p)^2]/p^2 < 1/(\nnum-1)$. This better (in fact, optimum) performance is obtained at the expense of some bias; namely, it is easily seen that \eqref{eq: hatvap nnum 2 vanden 1} gives $\E[\hatvap]/p = 1 - 1/(\nnum-1)$.

\subsection{A generalization of confidence}

According to \citet[proposition~3]{Mendo10}, for the loss function \eqref{eq: L conf}, given $\Omega \in \mathbb R^+$ and assuming that $\nnum \geq 3$, $\funo \geq {\Omega}/(\nnum - \sqrt{\nnum})$ and $\fdos \geq (\nnum + \sqrt\nnum + 1)/{\Omega}$, the estimator
\begin{equation}
\label{eq: hatvap Omega vanden mas 1}
\hatvap = \frac{\Omega}{\vanden+1}
\end{equation}
globally guarantees that $\risk(p)$ is smaller than its asymptotic value $\riskas$. Taking into account that, in this case, $\risk(p) = \Pr[\hatvap<p/\fdos]+\Pr[\hatvap>p\funo]$ and that the proof given in the cited reference considers the terms $\Pr[\hatvap<p/\fdos]$ and $\Pr[\hatvap>p\funo]$ separately, it can be seen that the same result holds for the loss function \eqref{eq: L conf gen} with $A_1=0$ or $A_2=0$. Furthermore, the result can be generalized to any loss function that can be approximated as a (possibly infinite) sum of functions of this form. This is the content of the next proposition.

\begin{proposition}
\label{prop: conf gen2 gar}
Given $\nnum \geq 3$ and $\Omega \in \mathbb R^+$, consider a loss function for which Assumptions~\ref{assum: L disc fin}, \ref{assum: L cota Theta} and \ref{assum: L crec} hold and that satisfies the following:
\begin{enumerate}
\item
\label{prop item: const centro}
$L$ is constant on an interval $[\upsilon, \upsilon']$, with
\begin{equation}
\label{eq: cond conf gen2 gar}
\upsilon \leq \frac{\Omega}{\nnum+\sqrt\nnum+1}, \quad
\upsilon' \geq \frac{\Omega}{\nnum-\sqrt\nnum}.
\end{equation}
\item
\label{prop item: no-inc izq}
$L$ is non-increasing on $(0,\upsilon]$.
\item
\label{prop item: non-dec der}
$L$ is non-decreasing on $[\upsilon',\infty)$.
\end{enumerate}
In these conditions, for any $p \in (0,1)$ the risk $\risk(p)$ of the estimator \eqref{eq: hatvap Omega vanden mas 1} satisfies $\risk(p) \leq \riskas$, with $\riskas$ given by \eqref{eq: riskas nu} (or \eqref{eq: riskas x}).
\end{proposition}

It is noted that conditions~\ref{prop item: const centro}--\ref{prop item: non-dec der}
of Proposition~\ref{prop: conf gen2 gar} imply that
Assumption~\ref{assum: L var acot} necessarily holds, and also imply that $L(\upsilon-) \geq L(\upsilon+)$ and $L(\upsilon'-) \leq L(\upsilon'+)$.

The following result, analogous to Corollary~\ref{cor: opt global gar NMSE}, is obtained for the estimator
\begin{equation}
\label{eq: hatvap Omega opt vanden mas 1}
\hatvap = \frac{\Omega^*}{\vanden+1}.
\end{equation}

\begin{corollary}
\label{cor: conf gen2 gar opt}
Given $\nnum \geq 3$ and a loss function that satisfies Assumptions~\ref{assum: L var acot}, \ref{assum: L disc fin}, \ref{assum: L cota Theta} and \ref{assum: L crec}, let $\Omega^*$ be as determined by Theorem~\ref{teo: min resp Omega}. If conditions~\ref{prop item: const centro}--\ref{prop item: non-dec der}
in Proposition~\ref{prop: conf gen2 gar} hold for some $\upsilon$, $\upsilon'$ with
\begin{equation}
\label{eq: cond conf gen2 gar opt}
\upsilon \leq \frac{\Omega^*}{\nnum+\sqrt\nnum+1}, \quad
\upsilon' \geq \frac{\Omega^*}{\nnum-\sqrt\nnum},
\end{equation}
the estimator \eqref{eq: hatvap Omega opt vanden mas 1} minimizes $\sup_{p \in (0,1)} \risk(p)$ among all (possibly randomized) estimators based on inverse binomial sampling.
\end{corollary}

This establishes that, under the stated hypotheses, the estimator \eqref{eq: hatvap Omega opt vanden mas 1} is minimax, i.e.~minimizes the risk that can be globally guaranteed not to be exceeded.

\appendix

\section{Proofs}
\label{parte: proofs}


The following definitions are necessary:
\begin{align}
\label{eq: Phi}
\Phi(p,\nu) & = \frac{(1-p)^{\nu/p-\nnum}}{(\nnum-1)!} \prod_{i=1}^{\nnum-1} (\nu-ip), \quad p \in (0,1), \ \nu \in \mathbb R^+, \\
\label{eq: riskas sop fin}
\riskasfs & = \int_{\nnum/\csopfin}^{\nnum \csopfin} \phi(\nu) L( \Omega/\nu ) \,\diff \nu, \quad \Omega, \csopfin \in \mathbb R^+.
\end{align}


\begin{lemma}[{\citet[lemma 1]{Mendo10}}]
\label{lemma: phi}
For any $\nu \in \mathbb R^+$, $0<\phi(\nu)<1$.
\end{lemma}

\begin{lemma}
\label{lemma: Phi conv unif}
Given $\nu_1, \nu_2 \in \mathbb R^+$ with $\nu_2 > \nu_1$, for $\nu \in [\nu_1,\nu_2]$ the function $\Phi(p,\nu)$ converges uniformly to $\phi(\nu)$ as $p \rightarrow 0$.
\end{lemma}

\begin{proof}
The lemma is equivalent to the result that $\Phi(p_k,\nu)$ converges uniformly on $\nu \in [\nu_1,\nu_2]$ for any sequence $(p_k)$ such that $p_k \in (0,1)$, $p_k \rightarrow 0$, which is proved by \citet[lemma 3]{Mendo10}.
\end{proof}

\begin{proof}[Proof of Theorem~\ref{teo: Feln incl Felp}]
The risk $\risk(p)$ tends to $\riskas$ for $p \rightarrow 0$ if and only if $\risk(p_k)$ converges to $\riskas$ for every sequence $(p_k)$ such that $p_k \in (0,1)$, $p_k \rightarrow 0$ \citep[theorem 4.12]{Apostol74}. Consider an arbitrary sequence of this type. Let $\risk^k = \risk(p_k)$, and let $f_k$ denote the probability function $f$ for $p=p_k$. Defining $\phi_k(\nu) = \Phi(p_k,\nu)$, it is seen from \eqref{eq: f} and \eqref{eq: Phi} that $f_k(\nden) = p_k \phi_k(\nden p_k)$.

From property~\ref{assum item: cota O, cero} in Assumption~\ref{assum: L cota O}, there exist $\Kcero \in \mathbb R$ and $\Mcero, \xcero \in \mathbb R^+$ such that
\begin{equation}
\label{eq: cota, cero}
L(x) < \Mcero x^\Kcero \quad \text{for } x < \xcero.
\end{equation}
Without loss of generality, it will be assumed that $\Kcero < 0$. On the other hand, property~\ref{assum item: cota O, inf} implies that there exist  $\Kinf < \nnum$ and $\Minf, \xinf \in \mathbb R^+$ such that
\begin{equation}
\label{eq: cota, inf}
L(x) < \Minf x^\Kinf \quad \text{for } x > \xinf.
\end{equation}

The risk $\risk^k$ is expressed from \eqref{eq: risk non-rand} as
\begin{equation}
\label{eq: risk 1}
\risk^k = \sum_{\nden=\nnum}^\infty f_k(\nden) L\left(\frac{\fest(\nden)}{p_k}\right).
\end{equation}
Given $\numenos, \numas \in \mathbb R^+$ with $\numas > \numenos$, let the set $I_k$ be defined as
\begin{equation}
\label{eq: I k}
I_k = \{\lfloor \numenos/p_k\rfloor, \lfloor \numenos/p_k\rfloor+1, \ldots, \lceil \numas/p_k\rceil\}.
\end{equation}
Under the assumption
\begin{equation}
\label{eq: p k numenos nnum}
p_k \leq \frac{\numenos}{\nnum},
\end{equation}
which implies that $\min I_k = \lfloor \numenos/p_k\rfloor \geq \nnum$, the following definition can be made:
\begin{equation}
\label{eq: risk0 1}
\risk^k_0 = \sum_{\nden \in I_k} f_k(\nden) L\left(\frac{\fest(\nden)}{p_k}\right).
\end{equation}
The proof will proceed as follows. With a suitable choice of $\numenos$ and $\numas$, and for $k$ sufficiently large, the term $\risk^k_0$ can be made arbitrarily close to $\riskas$, as will be seen. On the other hand, the difference $\risk^k - \risk^k_0$ will be decomposed as the sum of three terms, each of which can be made arbitrarily small for sufficiently large $k$. Adequate bounds will be derived for each of these four terms, and then the bounds will be suitably combined to show that $\risk^k$ tends to $\riskas$ as $k \rightarrow \infty$.

In the following, $\nden p_k$ will be denoted as $\ndenpk$. Assuming
\begin{equation}
\label{eq: p k numenos nnum 1}
p_k \leq \frac{\numenos}{\nnum+1},
\end{equation}
(which obviously implies \eqref{eq: p k numenos nnum}), it is easily seen that for $\nden \in I_k$, $\ndenpk$ is contained in the interval $I$ given as
\begin{equation}
I = \left[ \frac{\nnum\numenos}{\nnum+1}, \numas + \frac{\numenos}{\nnum+1} \right].
\end{equation}
Lemma~\ref{lemma: Phi conv unif} implies that the sequence of functions $(\phi_k)$ converges uniformly to $\phi$ for $\nu \in I$; that is, given $\epsilon_\epsf > 0$, there exists $k_\epsf$ such that $|\phi_k(\nu)-\phi(\nu)|<\epsilon_\epsf$ for $\nu \in I$, $k \geq k_\epsf$. Thus $f_k(\nden) = p_k\phi(\ndenpk) + p_k\theta_{\epsf,\nden}$ with $|\theta_{\epsf,\nden}|<\epsilon_\epsf$ for $\nden \in I_k$, $k \geq k_\epsf$. In these conditions, since $\phi(\ndenpk) > 0$ (Lemma~\ref{lemma: phi}), \eqref{eq: risk0 1} can be expressed as
\begin{equation}
\label{eq: risk0 2}
\risk^k_0 = \sum_{\nden \in I_k} p_k \phi(\ndenpk) \left( 1 + \frac{\theta_{\epsf,\nden}}{\phi(\ndenpk)} \right) L\left(\frac{\fest(\nden)}{p_k}\right).
\end{equation}
On the other hand, since $\nden \fest(\nden) \rightarrow \Omega$ as $\nden \rightarrow \infty$, given $\epsilon_\epsg>0$ there exists $\nden_\epsg \geq \nnum$ such that $|\nden \fest(\nden) - \Omega| < \epsilon_\epsg$ for all $\nden \geq \nden_\epsg$, i.e.~$\fest(\nden) = (\Omega+\theta_{\epsg,\nden})/\nden$ with $|\theta_{\epsg,\nden}|<\epsilon_\epsg$. Therefore, assuming
\begin{equation}
\label{eq: p k numenos nden epsg}
p_k \leq \frac{\numenos}{\nden_\epsg},
\end{equation}
which implies that $\min I_k \geq \nden_\epsg$, \eqref{eq: risk0 2} can be written as
\begin{equation}
\label{eq: risk0 2 bis}
\risk^k_0 = \sum_{\nden \in I_k} p_k \phi(\ndenpk) \left( 1 + \frac{\theta_{\epsf,\nden}}{\phi(\ndenpk)} \right) L\left( \frac{\Omega+\theta_{\epsg,\nden}}{\ndenpk} \right).
\end{equation}
Denoting $m_\phi = \min_{\nu \in I} \phi(\nu)$, which is non-zero because of Lemma~\ref{lemma: phi}, it stems from \eqref{eq: risk0 2 bis} that
\begin{equation}
\label{eq: risk0 3 bis}
\risk^k_0 = \left( 1 + \frac{\theta_{\epsf}}{m_\phi} \right) \sum_{\nden \in I_k} p_k\phi(\ndenpk) L\left( \frac{\Omega+\theta_{\epsg,\nden}}{\ndenpk} \right)
\end{equation}
for some $\theta_{\epsf}$ with $|\theta_{\epsf}|<\epsilon_\epsf$.

Assuming $\epsilon_\epsg \leq \Omega/2$, and taking into account \eqref{eq: p k numenos nnum 1}, it follows from \eqref{eq: I k} that for $\nden \in I_k$, both $\Omega/\ndenpk$ and $(\Omega+\theta_{\epsg,\nden})/\ndenpk$ are contained in the interval
\begin{equation}
I' = \left[ \frac{\Omega}{2(\numas+\numenos/(\nnum+1))}, \frac{3(\nnum+1) \Omega}{2\nnum\numenos} \right].
\end{equation}
According to Assumption~\ref{assum: L disc fin compacto},
$L$ has a finite number of discontinuities in $I'$. Let $d$ denote this number. Each of these discontinuities, located at $x_1,\ldots,x_d$, may be either a jump or a removable discontinuity. Let
\begin{equation}
J = \sum_{i=1}^d \left( \left| \lim_{x \rightarrow {x_i}^-} L(x) - L(x_i) \right| +
\left| \lim_{x \rightarrow {x_i}^+} L(x) - L(x_i) \right| \right).
\end{equation}
Thus $J$ represents the contribution of all discontinuities to the total variation of $L$ on $I'$.

The function $L$ on the interval $I'$ can be decomposed as the sum of a continuous function $L_\mathrm c$ and a piecewise constant function $L_\mathrm d$, the latter of which has discontinuities at $x_1, \ldots, x_d$. By the Heine-Cantor theorem \citep[theorem 4.47]{Apostol74}, $L_\mathrm c$ is uniformly continuous on $I'$. Since $|\theta_{\epsg,\nden}| < \epsilon_\epsg$, it follows that for any $\epsilon_\epsc>0$ there exists $\delta_\epsc$ such that $|L_\mathrm c((\Omega+\theta_{\epsg,\nden})/\ndenpk)-L_\mathrm c(\Omega/\ndenpk)| < \epsilon_\epsc$ for $\epsilon_\epsg < \delta_\epsc$, for all $\nden \in I_k$, and for all $k$. Regarding $L_\mathrm d$, let
\begin{equation}
U_k = \left\{
\nden \in I_k \st L_\mathrm d\left(\frac{\Omega+\theta_{\epsg,\nden}}{\ndenpk} \right) \neq L_\mathrm d\left(\frac{\Omega}{\ndenpk}\right)
\right\}.
\end{equation}
For $\nden \in I_k \setminus U_k$,
\begin{equation}
\label{eq: risk0 3 ter}
\left|L\left(\frac{\Omega+\theta_{\epsg,\nden}}{\ndenpk}\right)-L\left(\frac{\Omega}{\ndenpk}\right)\right| < \epsilon_\epsc.
\end{equation}
For each $\nden \in U_k$, $|L_\mathrm d((\Omega+\theta_{\epsg,\nden})/\ndenpk)-L_\mathrm d(\Omega/\ndenpk)|$ can be at at most $J$, and thus
\begin{equation}
\label{eq: risk0 3 quater}
\left|L\left(\frac{\Omega+\theta_{\epsg,\nden}}{\ndenpk}\right)-L\left(\frac{\Omega}{\ndenpk}\right)\right| < \epsilon_\epsc + J.
\end{equation}
Let $\chi_k$ denote the number of elements of $U_k$ divided by that of $I_k$. Taking into account that the latter is less than $(\numas - \numenos)/p_k+3 < (\numas - \numenos + 3)/p_k$ and that the function $\phi$ is upper-bounded by $1$ (Lemma~\ref{lemma: phi}), from \eqref{eq: risk0 3 ter} and \eqref{eq: risk0 3 quater} it follows that, for $\epsilon_\epsg < \delta_\epsc$,
\begin{multline}
\left| \sum_{\nden \in I_k} p_k\phi(\ndenpk) L\left( \frac{\Omega+\theta_{\epsg,\nden}}{\ndenpk} \right) -
\sum_{\nden \in I_k} p_k\phi(\ndenpk) L\left( \frac{\Omega}{\ndenpk} \right) \right| \\
< (\numas - \numenos + 3)(\epsilon_\epsc + J\chi_k).
\end{multline}
It is easily seen that $\lim_{k \rightarrow \infty}\chi_k$ can be made arbitrarily small by taking $\epsilon_\epsg$ sufficiently small. Thus, given $\epsilon_\epsd$, there exist $\delta_\epsd$, $k_\epsd$ such that $\chi_k < \epsilon_\epsd$ for $\epsilon_\epsg < \delta_\epsd$, $k \geq k_\epsd$. Consequently, for $\epsilon_\epsg < \min\{\delta_\epsc, \delta_\epsd\}$ and $k \geq k_\epsd$,
\begin{multline}
\label{eq: risk0 3 cinco}
\left| \sum_{\nden \in I_k} p_k\phi(\ndenpk) L\left( \frac{\Omega+\theta_{\epsg,\nden}}{\ndenpk} \right) -
\sum_{\nden \in I_k} p_k\phi(\ndenpk) L\left( \frac{\Omega}{\ndenpk} \right) \right| \\
< (\numas - \numenos + 3)(\epsilon_\epsc + J \epsilon_\epsd).
\end{multline}
From \eqref{eq: risk0 3 bis} and \eqref{eq: risk0 3 cinco},
\begin{equation}
\label{eq: risk0 4}
\risk^k_0 = \left( 1 + \frac{\theta_{\epsf}}{m_\phi} \right)
\left[
\sum_{\nden \in I_k} p_k\phi(\ndenpk) L\left( \frac{\Omega}{\ndenpk} \right)
+ (\numas - \numenos + 3)(\theta_\epsc+J\theta_\epsd)
\right]
\end{equation}
with $|\theta_\epsc| < \epsilon_\epsc$, $|\theta_\epsd| < \epsilon_\epsd$. The sum over $\nden$ in \eqref{eq: risk0 4} tends to $\int_{\numenos}^{\numas} \phi(\nu) L(\Omega/\nu) \,\diff\nu$ as $k \rightarrow \infty$. Thus for any $\epsilon_\epsi>0$ there exists $k_\epsi$ such that for all $k \geq k_\epsi$
\begin{equation}
\left| \sum_{\nden \in I_k} p_k\phi(\ndenpk) L\left( \frac{\Omega}{\ndenpk} \right) - \int_{\numenos}^{\numas} \phi(\nu) L\left( \frac{\Omega}{\nu} \right) \,\diff\nu \right| < \epsilon_\epsi,
\end{equation}
and therefore \eqref{eq: risk0 4} can be expressed for $k \geq \max\{k_\epsd,k_\epsi\}$ as
\begin{equation}
\label{eq: risk0 5}
\risk^k_0 = \left( 1 + \frac{\theta_{\epsf}}{m_\phi} \right)
\left[
\int_{\numenos}^{\numas} \phi(\nu) L\left(\frac \Omega \nu\right) \,\diff\nu + \theta_\epsi + (\numas - \numenos + 3)(\theta_\epsc+J\theta_\epsd)
\right]
\end{equation}
with $|\theta_\epsi|<\epsilon_\epsi$. In addition, given any $\epsilon_\epst$, there exist $\numenos_\epst$, $\numas_\epst$ with $\numas_\epst > \numenos_\epst$ such that $|\riskas - \int_{\numenos}^{\numas} \phi(\nu) L(\Omega/\nu) \,\diff\nu|<\epsilon_\epst$ for $0<\numenos \leq \numenos_\epst$, $\numas \geq \numas_\epst$. Thus, in these conditions,
\begin{equation}
\label{eq: risk0 6}
\risk^k_0 = \left( 1 + \frac{\theta_{\epsf}}{m_\phi} \right)
\left[
\riskas + \theta_\epst + \theta_\epsi + (\numas - \numenos + 3)(\theta_\epsc+J\theta_\epsd)
\right].
\end{equation}
with $|\theta_\epst|<\epsilon_\epst$.

The difference $\risk^k - \risk^k_0$ can be expressed as $\risk^k_1+\risk^k_2+\risk^k_3$, where
\begin{align}
\risk^k_1 & = \sum_{\nden=\nnum}^{\nden_\epsg-1} f_k(\nden) L\left(\frac{\fest(\nden)}{p_k}\right), \\
\label{eq: risk-hat 0}
\risk^k_2 & = \sum_{\nden=\nden_\epsg}^{\lfloor \numenos/p_k\rfloor-1} f_k(\nden) L\left(\frac{\fest(\nden)}{p_k}\right), \\
\label{eq: risk+ 0}
\risk^k_3 & = \sum_{\lceil \numas/p_k\rceil+1}^\infty f_k(\nden) L\left(\frac{\fest(\nden)}{p_k}\right).
\end{align}
Regarding the term $\risk^k_1$, from \eqref{eq: f} it is seen that
\begin{equation}
\label{eq: f k nden <}
f_k(\nden) < \frac{\nden^{\nnum-1}p_k^\nnum}{(\nnum-1)!}
\end{equation}
and therefore
\begin{equation}
\label{eq: risk-check 2}
0 < \risk^k_1 <
\sum_{\nden=\nnum}^{\nden_\epsg-1} \frac{\nden^{\nnum-1} p_k^\nnum}{(\nnum-1)!} L\left(\frac{\fest(\nden)}{p_k}\right) <
\frac{\nden_\epsg^{\nnum-1} p_k^\nnum}{(\nnum-1)!} \sum_{\nden=\nnum}^{\nden_\epsg-1} L\left(\frac{\fest(\nden)}{p_k}\right).
\end{equation}
The fact that $\lim_{\nden \rightarrow \infty} \nden \fest(\nden)$ exists and is finite implies that the function $\fest$ is upper-bounded by some constant $M_\mathrm g$. For $\fest(\nden)/p_k > \xinf$, \eqref{eq: cota, inf} implies that  $L(\fest(\nden)/p_k) < \Minf (M_\mathrm g/p_k)^\Kinf$. On the other hand, $\fest(\nden)/p_k$ in \eqref{eq: risk-check 2} is greater than $m_\mathrm g = \min\{\fest(\nnum),\fest(\nnum+1),\ldots,\fest(\nden_\epsg-1)\}$; and for $\fest(\nden)/p_k \in (m_\mathrm g, \xinf]$,
Assumption~\ref{assum: L var acot} implies that $L(\fest(\nden)/p_k)$ is lower than some value $M'_\mathrm g$, where both $m_\mathrm g$ and $M'_\mathrm g$ depend on $\nden_\epsg$. Thus, for the range of values of $\nden$ in \eqref{eq: risk-check 2},
\begin{equation}
\label{eq: risk-check 3}
L\left(\frac{\fest(\nden)}{p_k}\right) < \max \left\{ \frac{\Minf M_\mathrm g^\Kinf}{p_k^\Kinf}, M'_\mathrm g \right\} < \frac{\max\{\Minf M_\mathrm g^\Kinf, M'_\mathrm g\}}{p_k^\Kinf}.
\end{equation}
The sum in the right-most part of \eqref{eq: risk-check 2} is either empty or else it contains $\nden_\epsg-\nnum < \nden_\epsg$ terms. Therefore, using \eqref{eq: risk-check 3},
\begin{equation}
\label{eq: risk-check 4}
0 \leq \risk^k_1 <
\frac{\nden_\epsg^{\nnum} \max\{\Minf M_\mathrm g^\Kinf, M'_\mathrm g\}}{(\nnum-1)!} p_k^{\nnum-\Kinf}.
\end{equation}

Regarding $\risk^k_2$, the sum in \eqref{eq: risk-hat 0} is empty for $\numenos/p_k < \nden_\epsg+1$. If it is non-empty, since $\nden \geq \nden_\epsg$, the term $\fest(\nden)/p_k$ can be written as $(\Omega+\theta_{\epsg,\nden})/\ndenpk$ with $|\theta_{\epsg,\nden}|<\epsilon_\epsg$. Therefore, taking into account \eqref{eq: f k nden <},
\begin{equation}
\label{eq: risk-hat 2}
0 \leq \risk^k_2 <
\frac{p_k}{(\nnum-1)!} \sum_{\nden=\nden_\epsg}^{\lfloor \numenos/p_k\rfloor-1} \ndenpk^{\nnum-1} L \left( \frac{\Omega+\theta_{\epsg,\nden}}{\ndenpk} \right).
\end{equation}
Since $\epsilon_\epsg \leq \Omega/2$, it holds that $\Omega/2 < \Omega+\theta_{\epsg,\nden} < 3\Omega/2$, and thus for the range of values of $\nden$ in \eqref{eq: risk-hat 0}
\begin{equation}
\frac{3\Omega}{2\ndenpk} > \frac{\Omega+\theta_{\epsg,\nden}}{\ndenpk} > \frac{\Omega}{2\ndenpk} > \frac{\Omega}{2\numenos}.
\end{equation}
Therefore, assuming $\Omega/(2\numenos) \geq \xinf$, for $\nden$ within the indicated range it stems from \eqref{eq: cota, inf} that
\begin{equation}
\label{eq: risk-hat 3}
L \left( \frac{\Omega+\theta_{\epsg,\nden}}{\ndenpk} \right)
< \Minf \left( \frac{\Omega+\theta_{\epsg,\nden}}{\ndenpk} \right)^\Kinf
< \Minf \left( \frac{3\Omega}{2\ndenpk} \right)^\Kinf.
\end{equation}
Substituting \eqref{eq: risk-hat 3} into \eqref{eq: risk-hat 2},
\begin{equation}
\label{eq: risk-hat 3 bis}
0 \leq \risk^k_2 <
\frac{\Minf (3\Omega/2)^\Kinf p_k}{(\nnum-1)!} \sum_{\nden=\nden_\epsg}^{\lfloor \numenos/p_k\rfloor-1} \ndenpk^{\nnum-\Kinf-1}
< \frac{\Minf (3\Omega/2)^\Kinf}{(\nnum-1)!} \numenos^{\nnum-\Kinf}.
\end{equation}
Consider $\epsilon_\epst'>0$  arbitrary. Since $\Kinf<\nnum$, defining
\begin{equation}
\numenos_\epst' = \left( \frac{(\nnum-1)!\, \epsilon_\epst'}{\Minf (3\Omega/2)^\Kinf} \right)^{1/(\nnum-\Kinf)}
\end{equation}
it follows from \eqref{eq: risk-hat 3 bis} that for any $\numenos \leq \numenos_\epst'$
\begin{equation}
\label{eq: risk-hat 4}
0 \leq \risk^k_2 < \epsilon_\epst'.
\end{equation}

As for $\risk^k_3$, taking into account that $(1-p_k)^{1/p_k} < 1/e$, from \eqref{eq: f} and \eqref{eq: phi} it is seen that
$
f_k(\nden) < p_k \phi(\ndenpk) / (1-p_k)^\nnum.
$
In addition, \eqref{eq: p k numenos nden epsg} implies that $\nden \geq \nden_\epsg$ for any $\nden$ within the range in \eqref{eq: risk+ 0}. Thus
\begin{equation}
\label{eq: risk+ 1}
0 < \risk^k_3 < \frac {p_k}{(1-p_k)^\nnum} \sum_{\lceil \numas/p_k\rceil+1}^\infty \phi(\ndenpk) L\left(\frac{\Omega+\theta_{\epsg,\nden}}{\ndenpk}\right).
\end{equation}
Since $\epsilon_\epsg \leq \Omega/2$,
\begin{equation}
\frac{\Omega}{2\ndenpk} < \frac{\Omega+\theta_{\epsg,\nden}}{\ndenpk} < \frac{3\Omega}{2\ndenpk} < \frac{3\Omega}{2\numas}.
\end{equation}
Thus, assuming $3\Omega/(2\numas) < \xcero$, and taking into account that $\Kcero < 0$, it stems that for $\nden$ within the indicated range
\begin{equation}
\label{eq: risk+ L}
L\left( \frac{\Omega+\theta_{\epsg,\nden}}{\ndenpk} \right) < \Mcero \left(\frac{\Omega+\theta_{\epsg,\nden}}{\ndenpk}\right)^\Kcero <
\Mcero \left(\frac{\Omega}{2\ndenpk}\right)^\Kcero.
\end{equation}
If it is additionally assumed that $p_k \leq 1/2$, the factor $1/(1-p_k)^\nnum$ in \eqref{eq: risk+ 1} cannot exceed $2^\nnum$. Therefore
\begin{equation}
\label{eq: risk+ 2}
0 < \risk^k_3 < \frac{2^{\nnum-\Kcero} \Mcero \Omega^\Kcero}{(\nnum-1)!} \sum_{\lceil \numas/p_k\rceil+1}^\infty p_k \ndenpk^{\nnum-\Kcero-1} \exp(-\ndenpk).
\end{equation}
The sum in \eqref{eq: risk+ 2} tends to $\Gamma(\nnum-\Kcero,\numas)$ as $k\ \rightarrow \infty$. Thus, given $\epsilon_\epsi'>0$, there exists $k_\epsi'$ such that for $k \geq k_\epsi'$
\begin{equation}
0 < \sum_{\lceil \numas/p_k\rceil+1}^\infty p_k \ndenpk^{\nnum-\Kcero-1} \exp(-\ndenpk) < \Gamma(\nnum-\Kcero,\numas) + \epsilon_\epsi'.
\end{equation}
In addition, since $\Gamma(\nnum-\Kcero,\numas)$ is positive and tends to $0$ as $\numas \rightarrow \infty$, for any $\epsilon_\epst'' >0$ there exists $\numas_\epst''$ such that $0<\Gamma(\nnum-\Kcero,\numas)<\epsilon_\epst''$ for $\numas \geq \numas_\epst''$. Therefore \eqref{eq: risk+ 2} can be written as
\begin{equation}
\label{eq: risk+ 4}
0 < \risk^k_3 < \frac{2^{\nnum-\Kcero} \Mcero \Omega^\Kcero}{(\nnum-1)!} (\epsilon_\epst''+\epsilon_\epsi').
\end{equation}

To establish that $\risk^k \rightarrow \riskas$, it suffices to show that for any $\epsilon_0>0$, there exists $k_0$ such that $|\risk^k - \riskas| < \epsilon_0$ for all $k \geq k_0$. With the foregoing results, and taking into account the dependencies between the involved parameters, this is accomplished as follows. Given $\epsilon_0>0$, let
\begin{equation}
\label{eq: epst, eleg}
\epsilon_\epst = \frac{\epsilon_0}{\epscuant}.
\end{equation}
This determines the values $\numenos_\epst$ and $\numas_\epst$. Likewise, taking
\begin{equation}
\label{eq: epst', eleg}
\epsilon_\epst' = \frac{\epsilon_0}{\epscuant}
\end{equation}
determines $\numenos_\epst'$, and taking $\epsilon_\epst''$ such that
\begin{equation}
\label{eq: epst'' eleg}
\frac{2^{\nnum-\Kcero} \Mcero \Omega^\Kcero}{(\nnum-1)!} \epsilon_\epst'' = \frac{\epsilon_0}{\epscuant}
\end{equation}
determines $\numas_\epst''$. The values $\numenos$ and $\numas$ are selected as
\begin{align}
\label{eq: numenos}
\numenos & = \min\left\{ \numenos_\epst, \numenos_\epst', \frac{\Omega}{\xinf} \right\}, \\
\label{eq: numas}
\numas & = \max\left\{ \numas_\epst, \numas_\epst'', \frac{3\Omega}{2\xcero} \right\}.
\end{align}
(Note that, since $\numas_\epst > \numenos_\epst$, \eqref{eq: numenos} and \eqref{eq: numas} imply that $\numas > \numenos$.) From $\numenos$ and $\numas$, the intervals $I$ and $I'$ are obtained, and the values $m_\phi$, $d$ and $J$ can be computed. Taking
\begin{equation}
\label{eq: epsi, eleg}
\epsilon_\epsi = \frac{\epsilon_0}{\epscuant}
\end{equation}
determines $k_\epsi$. The parameter $\epsilon_\epsf$ is selected such that
\begin{equation}
\label{eq: epsf, eleg}
\left( \riskas + \frac{4\epsilon_0}{\epscuant} \right) \frac{\epsilon_\epsf}{m_\phi} =  \frac{\epsilon_0}{\epscuant},
\end{equation}
which determines $k_\epsf$. Next, $\epsilon_\epsc$ is chosen such that
\begin{equation}
\label{eq: epsc, eleg}
(\numas - \numenos + 3)\epsilon_\epsc = \frac{\epsilon_0}{\epscuant},
\end{equation}
from which $\delta_\epsc$ is obtained. Taking $\epsilon_\epsd$ as
\begin{equation}
\label{eq: epsd, eleg}
\epsilon_\epsd = \frac{\epsilon_\epsc}{J}
\end{equation}
determines $\delta_\epsd$ and $k_\epsd$. Choosing any $\epsilon_\epsg$ smaller than $\min\{\Omega/2,\delta_\epsc,\delta_\epsd\}$ determines $\nden_\epsg$, from which $m_\mathrm g$ and $M'_\mathrm g$ can be obtained. Let $k_\epsg$ be such that for all $k \geq k_\epsg$
\begin{equation}
\label{eq: epsg, eleg}
\frac{\nden_\epsg^{\nnum} \max\{\Minf M_\mathrm g^\Kinf, M'_\mathrm g\}}{(\nnum-1)!} p_k^{\nnum-\Kinf} < \frac{\epsilon_0}{\epscuant}.
\end{equation}
Let $k_\epsg'$ be chosen such that \eqref{eq: p k numenos nden epsg} holds for all $k \geq k_\epsg'$, and $k_\epsa$ such that \eqref{eq: p k numenos nnum 1} holds for all $k \geq k_\epsa$. The parameter $\epsilon_\epsi'$ is chosen as
\begin{equation}
\label{eq: epsi', eleg}
\epsilon_\epsi' = \epsilon_\epst'',
\end{equation}
which determines $k_\epsi'$. Finally, let $k_\epsh$ be such that $p_k \leq 1/2$ for all $k \geq k_\epsh$. Taking $k_0 = \max\{k_\epsi, k_\epsi', k_\epsf, k_\epsg, k_\epsg', k_\epsd, k_\epsa, k_\epsh\}$, the following inequalities are obtained for $k \geq k_0$. From \eqref{eq: risk0 6}, \eqref{eq: epst, eleg} and \eqref{eq: epsi, eleg}--\eqref{eq: epsd, eleg},
\begin{equation}
\label{eq: risk 0 res}
|\risk^k_0 - \riskas| < \frac{4\epsilon_0}{\epscuant} + \left(\riskas + \frac{4\epsilon_0}{\epscuant}\right) \frac{\epsilon_\epsf}{m_\phi}
= \frac{5\epsilon_0}{\epscuant};
\end{equation}
from \eqref{eq: risk-check 4} and \eqref{eq: epsg, eleg},
\begin{equation}
\label{eq: risk 1 res}
0 \leq \risk^k_1 < \frac{\epsilon_0}{\epscuant};
\end{equation}
from \eqref{eq: risk-hat 4} and \eqref{eq: epst', eleg},
\begin{equation}
\label{eq: risk 2 res}
0 \leq \risk^k_2 < \frac{\epsilon_0}{\epscuant};
\end{equation}
and from \eqref{eq: risk+ 4}, \eqref{eq: epst'' eleg} and \eqref{eq: epsi', eleg},
\begin{equation}
\label{eq: risk 3 res}
0 < \risk^k_3 < \frac{2\epsilon_0}{\epscuant}.
\end{equation}
Inequalities \eqref{eq: risk 0 res}--\eqref{eq: risk 3 res}
imply that $|\risk^k - \riskas| < \epsilon_0$ for all $k \geq k_0$, which concludes the proof.
\end{proof}

\begin{proof}[Proof of Proposition~\ref{prop: riskas C1 Omega}]
By
Assumption~\ref{assum: L disc fin}, let $D$ be the number of discontinuities of $L$, occurring at points $x_1 < x_2 < \cdots < x_D$. The asymptotic risk $\riskas$ can be expressed as $\sum_{i=0}^D \riskas_i$ with
\begin{align}
\label{eq: riskas trozo 0}
\riskas_0 & = \int_0^{x_1} \psi(x,\Omega) L(x) \,\diff x, \\
\label{eq: riskas trozo i}
\riskas_i & = \int_{x_i}^{x_{i+1}} \psi(x,\Omega) L(x) \,\diff x, \quad i=1,\ldots,D-1, \\
\label{eq: riskas trozo I}
\riskas_D & = \int_{x_D}^\infty \psi(x,\Omega) L(x) \,\diff x.
\end{align}
Given $i=1,\ldots,D-1$, let $L_i(x)$ be defined for $x \in [x_i,x_{i+1}]$ as
\begin{equation}
\label{eq: L i}
L_i(x) = \begin{cases}
L(x), & x_i < x < x_{i+1}, \\
L(x_i+), & x = x_i, \\
L(x_{i+1}-), & x = x_{i+1},
\end{cases}
\end{equation}
and let $\psiL_i(x,\Omega)$ be defined for $x \in [x_i,x_{i+1}]$, $\Omega \in \mathbb R^+$ as $\psiL_i(x,\Omega) = \psi(x,\Omega) L_i(x)$. Clearly, the integral in \eqref{eq: riskas trozo i} does not change if $\psi(x,\Omega) L(x)$ is replaced by $\psiL_i(x,\Omega)$. The function $\psiL_i$ is continuous on $[x_i,x_{i+1}] \times \mathbb R^+$, because it is the product of continuous functions. The function $\partial \psiL_i/\partial \Omega$ is similarly seen to be continuous. This implies \citep[corollary to theorem 5.9]{Fleming77} that $\riskas_i$ given by \eqref{eq: riskas trozo i} is a $C^1$ function of $\Omega$, with
\begin{equation}
\label{eq: der riskas trozo}
\frac{\diff \riskas_i}{\diff \Omega} =
\int_{x_i}^{x_{i+1}} \frac{\partial \psiL_i(x,\Omega)}{\partial \Omega} \,\diff x
= \int_{x_i}^{x_{i+1}} \frac{\partial \psi(x,\Omega)}{\partial \Omega} L(x) \,\diff x.
\end{equation}

Regarding $\riskas_0$, let $\psiL_0(x,\Omega) = \psi(x,\Omega) L(x)$ for $x \in (0,x_{i+1}]$, $\Omega \in \mathbb R^+$, and $\psiL_0(0,\Omega) = 0$. It is clear that $\psiL_0$ is continuous on $(0,x_1] \times \mathbb R^+$. In addition, its continuity at any point of the form $(0,\Omega_0)$ can be established as follows. Let $\Delta$ be any value such that $0<\Delta<\Omega_0$. For $\Omega \in (\Omega_0-\Delta,\Omega_0+\Delta)$ and $x>0$, $\psiL_0$ is bounded as
\begin{equation}
\label{eq: cota psi 0}
0 \leq \psiL_0(x,\Omega) < \frac{(\Omega_0+\Delta)^\nnum \exp(-(\Omega_0-\Delta)/x) L(x)}{x^{\nnum+1} (\nnum-1)!}.
\end{equation}
Property~\ref{assum item: cota O, cero} in Assumption~\ref{assum: L cota O} implies that the right-hand side of \eqref{eq: cota psi 0} tends to $0$ as $x \rightarrow 0$. Thus there exists $\delta>0$ such that $0 \leq \psiL_0(x,\Omega)<\epsilon$ for $0\leq x<\delta$, $|\Omega-\Omega_0|<\Delta$. This shows that $\psiL_0$ is continuous at $(0,\Omega_0)$, and thus on $[0,x_1] \times \mathbb R^+$. Using analogous arguments, $\partial \psiL_0/\partial \Omega$ can also be seen to be continuous on $[0,x_1] \times \mathbb R^+$. This implies that $\riskas_0$ is a $C^1$ function of $\Omega$, and \eqref{eq: der riskas trozo} holds for $i=0$ if the lower integration limit is replaced by $0$.

As for $\riskas_D$, let $\psiL(x,\Omega) = \psi(x,\Omega) L(x)$, and consider the function $\psiL(x,\Omega)/\Omega^\nnum$. This function and its partial derivative with respect to $\Omega$ are continuous on $(x_D,\infty) \times \mathbb R^+$, and satisfy the following bounds:
\begin{align}
\label{eq: cota psi}
0 < \frac{\psiL(x,\Omega)}{\Omega^\nnum} & < \frac{L(x)}{x^{\nnum+1} (\nnum-1)!}, \\
\label{eq: cota der psi}
0 > \frac{\partial (\psiL(x,\Omega)/\Omega^\nnum)}{\partial \Omega} & = -\frac{\exp(-\Omega/x)L(x)}{x^{\nnum+2}(\nnum-1)!} > -\frac{L(x)}{x^{\nnum+2}(\nnum-1)!}.
\end{align}
The right-most parts of \eqref{eq: cota psi} and \eqref{eq: cota der psi} are integrable on $(x_D,\infty)$, because of property~\ref{assum item: cota O, inf} in Assumption~\ref{assum: L cota O}. This implies \citep[theorem 5.9]{Fleming77} that $\riskas_D/\Omega^\nnum$ is a $C^1$ function of $\Omega$, and therefore so is $\riskas_D$; in addition,
$\diff \riskas_D/\diff \Omega$ satisfies an expression analogous to \eqref{eq: der riskas trozo} with the integration interval replaced by $(x_D,\infty)$.

The preceding results assure that $\diff \riskas / \diff \Omega = \sum_{i=0}^D \diff \riskas_i / \diff \Omega$ is continuous and can be expressed as in \eqref{eq: der riskas}. The equality \eqref{eq: der riskas Omega} readily follows from \eqref{eq: psi}, \eqref{eq: riskas x} and \eqref{eq: der riskas}.
\end{proof}


\begin{lemma}
\label{lemma: der int}
For any $a, c \in \mathbb R^+$, $b \in \mathbb R$,
\begin{equation}
\label{eq: lemma der int}
\frac{\diff}{\diff\Omega} \int_a^\infty \frac{\Omega^b \exp(-\Omega/x)}{x^{c+1}} \,\diff x
= \frac {\Omega^{b-1} \exp(-\Omega/a)} {a^c} + (b-c) \Omega^{b-c-1} \gamma \left( c,\frac \Omega a \right) .
\end{equation}
\end{lemma}

\begin{proof}
Applying the change of variable $x=\Omega/\nu$, the integral in \eqref{eq: lemma der int} can be expressed as
\begin{equation}
\int_a^\infty \frac{\Omega^b \exp(-\Omega/x)}{x^{c+1}} \,\diff x
= \int_0^{\Omega/a} \Omega^{b-c} \nu^{c-1} \exp(-\nu) \,\diff \nu
= \Omega^{b-c} \left( c,\frac \Omega a \right) ,
\end{equation}
from which \eqref{eq: lemma der int} follows.
\end{proof}

\begin{lemma}
For $s \in \mathbb R$,
\label{lemma: lim gamma}
\begin{align}
\label{eq: lim gamma 0}
\lim_{u \rightarrow 0} \frac{\gamma(s,u)}{u^s} & = \frac 1 s, \\
\label{eq: lim gamma inf}
\lim_{u \rightarrow \infty} \frac{\Gamma(s,u)}{u^{s-1} \exp(-u)} & = 1.
\end{align}
\end{lemma}

\begin{proof}
These equalities respectively follow from \citet[equation 6.5.29]{Abramowitz70} and \citet[equation 6.5.32]{Abramowitz70}.
\end{proof}

\begin{lemma}
\label{lemma: Gamma}
The upper incomplete gamma function
\eqref{eq: Gamma} satisfies for $s, w \in \mathbb N$, $\nu \in \mathbb R^+$
\begin{equation}
\Gamma(s,\nu) =
\begin{cases}
\displaystyle
\ \,\ \sum_{k=0}^{s-1} \ff{(s-1)}{s-k-1} \nu^k \exp(-\nu), & s \geq 1, \\
\displaystyle
\sum_{k=s-w}^{s-1} \ff{(s-1)}{s-k-1} \nu^k \exp(-\nu) + W(\nu), & s \leq 0,
\end{cases}
\end{equation}
where $W(\nu)$ is $O(\nu^{s-w-1} \exp(-\nu))$ as $\nu \rightarrow \infty$.
\end{lemma}

\begin{proof}
The expression for $s \geq 1$ is equivalent to \citet[equation 6.5.13]{Abramowitz70}.

For $s \leq 0$, the stated result follows from recursively using the identity \citep[equation 6.5.21]{Abramowitz70}
\begin{equation}
\Gamma(s,\nu) = (s-1) \Gamma(s-1,\nu) + \nu^{s-1} \exp(-\nu)
\end{equation}
$w$ times and taking into account the equality \eqref{eq: lim gamma inf} from Lemma~\ref{lemma: lim gamma}.
\end{proof}

\begin{lemma}
\label{lemma: der binom partic}
For $t \in \mathbb N$, $u \in \mathbb Z$,
\begin{equation}
\label{eq: lemma der binom partic}
\sum_{j=0}^t \binom{t}{j} j \ff{(u-j)}{i-1} (-1)^{t-j} =
\begin{cases}
\displaystyle
0, & i =1,\ldots, t-1, \\
\displaystyle
(-1)^{t-1} t!, & i = t.
\end{cases}
\end{equation}
\end{lemma}

\begin{proof}
The equality
\begin{equation}
\label{eq: lemma der binom partic demo 1}
\sum_{j=0}^t \binom{t}{j} \ff{j}{k} (-1)^{t-j} =
\begin{cases}
\displaystyle
0, & k \neq t, \\
\displaystyle
t!, & k = t.
\end{cases}
\end{equation}
is easily shown to hold for $k \in \mathbb N$ by applying the binomial theorem to $(x-1)^t$, differentiating $k$ times and particularizing for $x=1$. The term $j \ff{(u-j)}{i-1}$ in \eqref{eq: lemma der binom partic} can be expressed as $\sum_{k=1}^i a_k \ff{j}{k}$ for appropriate values of the coefficients $a_k$; furthermore, it is easily seen that $a_i$ equals $(-1)^{i-1}$. Thus
\begin{equation}
\label{eq: lemma der binom partic demo 3}
\sum_{j=0}^t \binom{t}{j} j \ff{(u-j)}{i-1} (-1)^{t-j} =
\sum_{k=1}^i \sum_{j=0}^t a_k \binom{t}{j} \ff{j}{k} (-1)^{t-j}.
\end{equation}
If $i \leq t-1$, the inner sum in \eqref{eq: lemma der binom partic demo 3} equals $0$ for all $k$ within the range specified in the outer sum, because of \eqref{eq: lemma der binom partic demo 1}. If $i=t$, all values of the index $k$ give a null inner sum except $k=t$, which gives $a_t t!$ = $(-1)^{t-1} t!$. This establishes \eqref{eq: lemma der binom partic}.
\end{proof}

\begin{proof}[Proof of Proposition~\ref{prop: assum item bien izq, lim}]
Assume that \eqref{eq: cond 0, lim} holds. Let $\epsilon = -Bs/(4\nnum)$, which is positive for the allowed values of $B$ and $s$. From \eqref{eq: cond 0, lim}, there exists $\delta$ such that $|L(x)-A-Bx^s| < \epsilon x^s$ for all $x \in (0,\delta)$. This implies that for any $\xcerocr \in (0,\delta)$, and for $\xcerocr \leq x < \delta$,
\begin{equation}
L(\xcerocr)-L(x) > B(\xcerocr^s-x^s) - 2\epsilon x^s = B\xcerocr^s - (B+2\epsilon)x^s.
\end{equation}
Therefore
\begin{equation}
\label{eq: cond 0, lim demo 1}
\begin{split}
\quad & \int_{\xcerocr}^\infty \frac{L(\xcerocr)-L(x)}{x^{\nnum+1}} \,\diff x \\
& = \int_{\xcerocr}^\delta \frac{L(\xcerocr)-L(x)}{x^{\nnum+1}} \,\diff x + \int_\delta^\infty \frac{L(\xcerocr)-L(x)}{x^{\nnum+1}} \,\diff x \\
& > B \xcerocr^s \int_{\xcerocr}^\delta \frac{\diff x}{x^{\nnum+1}}
- (B+2\epsilon) \int_{\xcerocr}^\delta \frac{\diff x}{x^{\nnum-s+1}}
+ \int_\delta^\infty \frac{L(\xcerocr)-L(x)}{x^{\nnum+1}} \,\diff x \\
& = \frac{B}{\nnum \xcerocr^{\nnum-s}} - \frac{B \xcerocr^s}{\nnum \delta^\nnum} + \frac{B+2\epsilon}{(\nnum-s)\delta^{\nnum-s}} - \frac{B+2\epsilon}{(\nnum-s)\xcerocr^{\nnum-s}} + \int_\delta^\infty \frac{L(\xcerocr)-L(x)}{x^{\nnum+1}} \,\diff x \\
& > \frac{B}{\nnum \xcerocr^{\nnum-s}} - \frac{B}{\nnum \delta^{\nnum-s}} + \frac{B+2\epsilon}{(\nnum-s)\delta^{\nnum-s}} - \frac{B+2\epsilon}{(\nnum-s)\xcerocr^{\nnum-s}} + \int_\delta^\infty \frac{L(\xcerocr)-L(x)}{x^{\nnum+1}} \,\diff x
\end{split}
\end{equation}
Denoting by $C$ the sum of the terms in the right-hand side of \eqref{eq: cond 0, lim demo 1} which do not depend on $\xcerocr$, i.e.~the second, third and fifth, and substituting the value of $\epsilon$,
\begin{equation}
\label{eq: cond 0, lim demo 2}
\int_{\xcerocr}^\infty \frac{L(\xcerocr)-L(x)}{x^{\nnum+1}} \,\diff x >
- \frac{Bs}{2\nnum(\nnum-s)\xcerocr^{\nnum-s}} + C.
\end{equation}
Taking into account that $-Bs$ and $\nnum-s$ are positive, and that $C$ is independent of $\xcerocr$, from \eqref{eq: cond 0, lim demo 2} it is seen that there exists $\xcerocr \in (0,\delta)$ such that \eqref{eq: cond 0, integral} holds.
\end{proof}

\begin{lemma}
\label{lemma: min resp Omega izq}
Under the hypotheses of Theorem~\ref{teo: min resp Omega}, there exists $\Omega_0$ such that $\diff \riskas/\diff \Omega < 0$ for all $\Omega \leq \Omega_0$.
\end{lemma}

\begin{proof}
Let $\xcerocr$ be as in property~\ref{assum item: bien izq} in Assumption~\ref{assum: L crec}. Since $L$ is non-increasing for all $x$ smaller than $\xcerocr$, the function $\ell$ defined as
\begin{equation}
\ell(x) =
\begin{cases}
L(x)-L(\xcerocr) & \text{for } 0 < x < \xcerocr \\
0 & \text{for } x \geq \xcerocr
\end{cases}
\end{equation}
is non-negative and non-increasing. From \eqref{eq: riskas nu} and \eqref{eq: riskas x}, $\riskas$ can be expressed as $\riskascero_0 + \riskascero_1 + \riskascero_2$ with
\begin{align}
\label{eq: riskas 0}
\riskascero_0 & = \int_{\Omega/\xcerocr}^\infty \phi(\nu) L( \xcerocr ) \,\diff \nu ,\\
\label{eq: riskas 0'}
\riskascero_1 & = \int_{\Omega/\xcerocr}^\infty \phi(\nu) \ell( \Omega/\nu ) \,\diff \nu, \\
\label{eq: riskas 0''}
\riskascero_2 & = \int_{\xcerocr}^\infty \psi(x,\Omega) L(x) \,\diff x.
\end{align}
Each of these terms can be interpreted as the risk associated with a certain loss function for which Proposition~\ref{prop: riskas C1 Omega} applies.

Since $\ell$ is non-negative and non-increasing, for $\nu$ fixed the integrand in \eqref{eq: riskas 0'} is a non-negative, non-increasing function of $\Omega$. This implies that $\riskascero_1$ is a non-increasing function of $\Omega$, and thus $\diff \riskascero_1/\diff\Omega \leq 0$.

Regarding the term $\riskascero_0$,
\begin{equation}
\frac{\diff \riskascero_0}{\diff \Omega} = -\frac{\Omega^{\nnum-1} \exp(-\Omega/\xcerocr) L(\xcerocr)}{\xcerocr^\nnum(\nnum-1)!},
\end{equation}
which implies that
\begin{equation}
\label{eq: lim der riskas 0}
\lim_{\Omega \rightarrow 0} \frac{\diff \riskascero_0/\diff \Omega}{\Omega^{\nnum-1}}  =
- \frac{L(\xcerocr)}{\xcerocr^\nnum(\nnum-1)!}.
\end{equation}

As for $\riskascero_2$, from \eqref{eq: riskas 0''} it follows that
\begin{equation}
\label{eq: der riskas 1 x dos trozos}
\frac{\diff \riskascero_2}{\diff \Omega} =
\frac{\nnum \Omega^{\nnum-1}}{(\nnum-1)!} \int_{\xcerocr}^\infty \frac{\exp(-\Omega/x) L(x)}{x^{\nnum+1}} \,\diff x
- \frac{\Omega^\nnum}{(\nnum-1)!} \int_{\xcerocr}^\infty \frac{\exp(-\Omega/x) L(x)}{x^{\nnum+2}} \,\diff x.
\end{equation}
Interpreting the integrals in \eqref{eq: der riskas 1 x dos trozos} as Lebesgue integrals, and noting that $\exp(-\Omega/x) < 1$ for $\Omega, x \in \mathbb R^+$, Lebesgue's dominated convergence theorem \citep[theorem 10.27]{Apostol74} assures that
\begin{equation}
\lim_{\Omega \rightarrow 0} \int_{\xcerocr}^\infty \frac{\exp(-\Omega/x) L(x)}{x^{\nnum+1}} \,\diff x = \int_{\xcerocr}^\infty \frac{L(x)}{x^{\nnum+1}} \,\diff x,
\end{equation}
and similarly for the second integral. This implies that the first term in the right-hand side of \eqref{eq: der riskas 1 x dos trozos} dominates the second for $\Omega$ asymptotically small, i.e.
\begin{equation}
\label{eq: lim der riskas 0''}
\lim_{\Omega \rightarrow 0} \frac{\diff \riskascero_2/\diff \Omega}{\Omega^{\nnum-1}} =
\frac{\nnum}{(\nnum-1)!} \int_{\xcerocr}^\infty \frac{L(x)}{x^{\nnum+1}} \,\diff x.
\end{equation}

From \eqref{eq: lim der riskas 0} and \eqref{eq: lim der riskas 0''},
\begin{equation}
\label{eq: lim der riskas}
\begin{split}
\lim_{\Omega \rightarrow 0} \frac{\diff (\riskascero_0 + \riskascero_2)/\diff \Omega}{\Omega^{\nnum-1}}
& = - \frac{L(\xcerocr)}{\xcerocr^\nnum(\nnum-1)!} + \frac{\nnum}{(\nnum-1)!} \int_{\xcerocr}^\infty \frac{L(x)}{x^{\nnum+1}} \,\diff x \\
& = - \frac{\nnum}{(\nnum-1)!} \int_{\xcerocr}^\infty \frac{L(\xcerocr)-L(x)}{x^{\nnum+1}} \,\diff x.
\end{split}
\end{equation}
Combining \eqref{eq: lim der riskas} with the inequality \eqref{eq: cond 0, integral} from Assumption~\ref{assum: L crec}, the limit on the right-hand side of \eqref{eq: lim der riskas} is seen to be negative. This implies that there exists $\Omega_0$ such that $\diff(\riskascero_0+\riskascero_2)/\diff\Omega < 0$ for $\Omega \leq  \Omega_0$. Taking into account that $\diff \riskascero_1/\diff\Omega \leq 0$, it follows that $\diff \riskas/\diff\Omega < 0$ for $\Omega \leq \Omega_0$.
\end{proof}

\begin{lemma}
\label{lemma: min resp Omega der}
Under the hypotheses of Theorem~\ref{teo: min resp Omega}, there exists $\Omega_0'$ such that $\diff \riskas/\diff \Omega > 0$ for all $\Omega \geq \Omega_0'$.
\end{lemma}

\begin{proof}
If condition~\ref{assum item item: i} of property~\ref{assum item: bien der} in Assumption~\ref{assum: L crec} holds, let $H$ be chosen such that $0<H<L(\xinfcr+)-L(\xinfcr-)$. By definition of $L(\xinfcr-)$, there exists $h$ such that $L(x) \in (L(\xinfcr-)-H,L(\xinfcr-)+H)$ for all $x \in (\xinfcr-h,\xinfcr)$. If condition~\ref{assum item item: ii} holds, it stems that there exists $h$ such that $(-1)^{t-1} \diff^t L / \diff x^t$ is positive and continuous for $x \in (\xinfcr-h,\xinfcr)$. Thus, let $h$ be selected as has been indicated.

From property~\ref{assum item: cota O, cero} in Assumption~\ref{assum: L cota O}, there exist $\Kcero \in \mathbb R$, $\Mcero$ and $\xcero < \xinfcr-h$ such that
\begin{equation}
\label{eq: cota, cero, otra vez}
L(x) < \Mcero x^\Kcero \quad \text{for } x < \xcero.
\end{equation}
The asymptotic risk $\riskas$ can be expressed from \eqref{eq: riskas nu} and \eqref{eq: riskas x} as
$\riskasinf_0 + \riskasinf_1 + \riskasinf_2 + \riskasinf_3+\riskasinf_4$ with
\begin{align}
\label{eq: riskas 1}
\riskasinf_0 & = \int_{\xinfcr-h}^{\xinfcr} \psi(x,\Omega) L(x) \,\diff x, \\
\riskasinf_1 & = \int_0^\xcero \psi(x,\Omega) L(x) \,\diff x, \\
\riskasinf_2 & = \int_\xcero^{\xinfcr-h} \psi(x,\Omega) L(x) \,\diff x, \\
\riskasinf_3 & = \int_{\xinfcr}^\infty \psi(x,\Omega) L(\xinfcr+) \,\diff x, \\
\label{eq: riskasinf 4}
\riskasinf_4 & = \int_0^{\Omega/\xinfcr} \phi(\nu) (L(\Omega/\nu)-L(\xinfcr+)) \,\diff \nu.
\end{align}
Each of these terms corresponds to the risk associated with a certain loss function which satisfies Proposition~\ref{prop: riskas C1 Omega}.

By property~\ref{assum item: bien der} of Assumption~\ref{assum: L crec}, $L(x)-L(\xinfcr+)$ is non-negative and non-decreasing for $x > \xinfcr$. An argument analogous to that used for $\riskascero_1$ in Lemma~\ref{lemma: min resp Omega izq} shows that the term $\riskasinf_4$ given by \eqref{eq: riskasinf 4} is non-decreasing with $\Omega$, and thus
\begin{equation}
\label{eq: der riskas 1''''}
\frac{\diff\riskasinf_4}{\diff\Omega} \geq 0.
\end{equation}

According to Lemma~\ref{lemma: der int}, $\diff\riskasinf_3/\diff\Omega$ is given by
\begin{equation}
\label{eq: der riskas 1'''}
\frac{\diff\riskasinf_3}{\diff\Omega} = \frac{L(\xinfcr+) \Omega^{\nnum-1} \exp(-\Omega/\xinfcr)}{\xinfcr^\nnum (\nnum-1)!}.
\end{equation}

Computing
\begin{equation}
\frac{\diff\riskasinf_1}{\diff\Omega} = \int_0^\xcero \frac {\nnum \Omega^{\nnum-1} \exp(-\Omega/x)} {x^{\nnum+1} (\nnum-1)!} L(x) \,\diff x
- \int_0^\xcero \frac {\Omega^\nnum \exp(-\Omega/x)} {x^{\nnum+2} (\nnum-1)!} L(x) \,\diff x
\end{equation}
and using \eqref{eq: cota, cero, otra vez} it stems that
\begin{equation}
\label{eq: abs riskas 1'}
\left| \frac{\diff\riskasinf_1}{\diff\Omega} \right| \leq \frac{\Mcero \nnum \Omega^{\nnum-1}}{(\nnum-1)!} \int_0^\xcero \frac {\exp(-\Omega/x)} {x^{\nnum+1-\Kcero}} \,\diff x
+ \frac{\Mcero \Omega^{\nnum}}{(\nnum-1)!} \int_0^\xcero \frac {\exp(-\Omega/x)} {x^{\nnum+2-\Kcero}} \,\diff x.
\end{equation}
The integrals in \eqref{eq: abs riskas 1'} can be bounded as follows. Let $\lambda = (\xcero + \xinfcr-h)/(2(\xinfcr-h))$. It is seen that $\lambda$ and $1-\lambda$ are lower than $1$. Let the function $v_1: \mathbb R^+ \cup \{0\} \mapsto \mathbb R \cup \{0\}$ be defined as $v_1(x) = \exp(-\lambda\Omega/x)$ for $x > 0$ and $v_1(0)=0$. Since $\exp(-\lambda\Omega/x) \rightarrow 0$ as $x \rightarrow 0$, $v_1$ is continuous on $[0, \xcero]$. In addition, the function $v_2: \mathbb R \cup \{0\} \mapsto \mathbb R \cup \{0\}$ such that
\begin{equation}
v_2(x) = \frac{\exp(-(1-\lambda)\Omega/x)} {x^{\nnum+1-\Kcero}}
\end{equation}
for $x>0$ and $v_2(0)=0$ is non-negative and integrable on $[0, \xcero]$. Thus, the mean value theorem \citep[p.~190]{Fleming77} can be applied to the first integral in \eqref{eq: abs riskas 1'} to yield:
\begin{equation}
\label{eq: TVM int 1}
\int_0^\xcero \frac {\exp(-\Omega/x)} {x^{\nnum+1-\Kcero}} \,\diff x = \int_0^\xcero v_1(x) v_2(x) \,\diff x
= v_1(x_\mathrm m) \int_0^\xcero v_2(x) \,\diff x
\end{equation}
for some $x_\mathrm m \in [0, \xcero]$. Actually, $x_\mathrm m$ cannot be $0$, because that would give $0$ in the right-hand side of \eqref{eq: TVM int 1}, whereas the left-hand side is greater than $0$. Thus $x_\mathrm m \in (0, \xcero]$. Similar arguments can be applied to the last integral in \eqref{eq: TVM int 1} to obtain
\begin{equation}
\label{eq: TVM int 1bis}
\int_0^\xcero \frac {\exp(-(1-\lambda)\Omega/x)} {x^{\nnum+1-\Kcero}} \,\diff x = \xcero \frac {\exp(-(1-\lambda)\Omega/x'_\mathrm m)} {{x'_\mathrm m}^{\nnum+1-\Kcero}}
\end{equation}
with $x'_\mathrm m \in (0, \xcero]$. Maximizing the right-hand side of \eqref{eq: TVM int 1bis} with respect to $x'_\mathrm m \in \mathbb R^+$ gives
\begin{equation}
\label{eq: TVM int 1ter}
\int_0^\xcero \frac {\exp(-(1-\lambda)\Omega/x)} {x^{\nnum+1-\Kcero}} \,\diff x \leq \xcero
\left( \frac {\nnum+1-\Kcero} {1-\lambda} \right)^{\nnum+1-\Kcero}
\frac { \exp(-(\nnum+1-\Kcero))} {\Omega^{\nnum+1-\Kcero}}.
\end{equation}
Combining \eqref{eq: TVM int 1} and \eqref{eq: TVM int 1ter},
\begin{equation}
\label{eq: TVM int 1cua}
\int_0^\xcero \frac {\exp(-\Omega/x)} {x^{\nnum+1-\Kcero}} \,\diff x \leq \frac {\xcero (\nnum+1-\Kcero)^{\nnum+1-\Kcero} \exp(-(\nnum+1-\Kcero+\lambda\Omega/x_\mathrm m))} {((1-\lambda)\Omega)^{\nnum+1-\Kcero}}.
\end{equation}
The second integral in \eqref{eq: abs riskas 1'} is bounded analogously:
\begin{equation}
\label{eq: TVM int 1cin}
\int_0^\xcero \frac {\exp(-\Omega/x)} {x^{\nnum+2-\Kcero}} \,\diff x \leq \frac {\xcero (\nnum+2-\Kcero)^{\nnum+2-\Kcero} \exp(-(\nnum+2-\Kcero+\lambda\Omega/x''_\mathrm m))} {((1-\lambda)\Omega)^{\nnum+2-\Kcero}}.
\end{equation}
with $x''_\mathrm m \in (0, \xcero]$. From \eqref{eq: abs riskas 1'}, \eqref{eq: TVM int 1cua} and \eqref{eq: TVM int 1cin},
\begin{equation}
\label{eq: TVM int 1sei}
\begin{split}
\left| \frac{\diff\riskasinf_1}{\diff\Omega} \right| & \leq
\frac {\Mcero \xcero \nnum (\nnum+1-\Kcero)^{\nnum+1-\Kcero} \Omega^{\Kcero-2} \exp(-(\nnum+1-\Kcero+\lambda\Omega/x_\mathrm m))} {(1-\lambda)^{\nnum+1-\Kcero} (\nnum-1)!} \\
& \quad + \frac {\Mcero \xcero (\nnum+2-\Kcero)^{\nnum+2-\Kcero} \Omega^{\Kcero-2} \exp(-(\nnum+2-\Kcero+\lambda\Omega/x''_\mathrm m))} {(1-\lambda)^{\nnum+2-\Kcero} (\nnum-1)!}.
\end{split}
\end{equation}
It is easily seen that $x_\mathrm m/\lambda, x''_\mathrm m/\lambda < \xcero < \xinfcr - h$. It thus follows from \eqref{eq: TVM int 1sei} that
\begin{equation}
\label{eq: der riskas 1'}
\left| \frac{\diff\riskasinf_1}{\diff\Omega} \right| < Q \Omega^{\Kcero-2} \exp(-\Omega/(\xinfcr-h))
\end{equation}
where $Q$ is independent of $\Omega$.

For $\diff\riskasinf_2/\diff\Omega$, by Assumption~\ref{assum: L var acot}, let $M$ be an upper bound of $L$ in the interval $(\xcero, \xinfcr-h)$. An argument based on the mean value theorem can also be applied here; in fact, it is slightly simpler than in the preceding paragraph because in this case the lower integration limit is greater than $0$:
\begin{equation}
\begin{split}
\left| \frac{\diff\riskasinf_2}{\diff\Omega} \right| & \leq
\frac{M \nnum \Omega^{\nnum-1}}{(\nnum-1)!} \int_\xcero^{\xinfcr-h} \frac{\exp(-\Omega/x)}{x^{\nnum+1}} \,\diff x
+ \frac{M \Omega^\nnum}{(\nnum-1)!} \int_\xcero^{\xinfcr-h} \frac{\exp(-\Omega/x)}{x^{\nnum+2}} \,\diff x \\
& = \frac { M \Omega^{\nnum-1} (\xinfcr-h-\xcero) } { (\nnum-1)! } \left( \frac{\nnum \exp(-\Omega/x'''_\mathrm m)}{{x'''_\mathrm m}^{\nnum+1}} + \frac{\Omega \exp(-\Omega/x''''_\mathrm m)}{{x''''_\mathrm m}^{\nnum+2}} \right) \\
\end{split}
\end{equation}
with $x'''_\mathrm m, x''''_\mathrm m \in [\xcero, \xinfcr-h]$. Therefore
\begin{equation}
\label{eq: der riskas 1''}
\left| \frac{\diff\riskasinf_2}{\diff\Omega} \right| < \frac { M (\xinfcr-h-\xcero) } { {\xcero}^{\nnum+1} (\nnum-1)! } \left( \nnum + \frac{\Omega}{\xcero} \right) \Omega^{\nnum-1} \exp(-\Omega/(\xinfcr-h)).
\end{equation}

To compute the derivative of $\riskasinf_0$, it is necessary to distinguish cases~\ref{assum item item: i} and \ref{assum item item: ii} of property~\ref{assum item: bien der} in Assumption~\ref{assum: L crec}. In case~\ref{assum item item: i}, since $L(x) \in (L(\xinfcr-)-H, L(\xinfcr-)+H)$ for all $x \in (\xinfcr-h,\xinfcr)$, the mean value theorem assures that there is some $\theta \in [L(\xinfcr-)-H,L(\xinfcr-)+H]$ such that
\begin{equation}
\begin{split}
\frac{\diff\riskasinf_0}{\diff\Omega} & = \int_{\xinfcr-h}^{\xinfcr} \frac{\partial \psi(x,\Omega)}{\partial \Omega} L(x) \,\diff x
= \theta \int_{\xinfcr-h}^{\xinfcr} \frac{\partial \psi(x,\Omega)}{\partial \Omega} \,\diff x \\
& = \frac{\theta}{(\nnum-1)!} \frac{\diff}{\diff\Omega}  \int_{\xinfcr-h}^{\xinfcr} \frac {\Omega^\nnum \exp(-\Omega/x)} {x^{\nnum+1} } \,\diff x.
\end{split}
\end{equation}
Applying Lemma~\ref{lemma: der int},
\begin{equation}
\label{eq: der princ grado 0}
\frac{\diff\riskasinf_0}{\diff\Omega} = \frac{\theta \Omega^{\nnum-1}}{(\nnum-1)!} \left( -\frac{\exp(-\Omega/\xinfcr)}{\xinfcr^\nnum} + \frac{\exp(-\Omega/(\xinfcr-h))}{(\xinfcr-h)^\nnum} \right).
\end{equation}
Using \eqref{eq: der riskas 1''''}, \eqref{eq: der riskas 1'''}, \eqref{eq: der riskas 1'}, \eqref{eq: der riskas 1''} and \eqref{eq: der princ grado 0},
\begin{equation}
\label{eq: der Omega infty grado 0}
\frac{\diff\riskas}{\diff\Omega} \geq \frac{(L(\xinfcr+)-\theta) \Omega^{\nnum-1}\exp(-\Omega/\xinfcr)}
{\xinfcr^\nnum (\nnum-1)!}
+ O \left( \Omega^q \exp(-\Omega/(\xinfcr-h)) \right)
\end{equation}
with $q = \max\{\nnum,\Kcero-2\}$. Since $h>0$ and $\theta \leq L(\xinfcr-)+H<L(\xinfcr+)$, from \eqref{eq: der Omega infty grado 0} it follows that
\begin{equation}
\label{eq: lim der Omega infty grado 0}
\lim_{\Omega \rightarrow \infty}  \left( \frac{\exp(\Omega/\xinfcr)}{\Omega^{\nnum-1}} \frac{\diff\riskas}{\diff\Omega} \right) \geq \frac{L(\xinfcr+)-\theta}{\xinfcr^\nnum (\nnum-1)!} > 0.
\end{equation}

In case~\ref{assum item item: ii}, since $\diff^t L/\diff x^t$ is continuous on $(\xinfcr-h,\xinfcr)$, Taylor's theorem \citep[volume 1, theorem 7.6]{Apostol67} can be applied to express $L(x)$ for $x \in (\xinfcr-h,\xinfcr)$ as
\begin{equation}
\label{eq: Taylor}
L(x) = L(\xinfcr-) + \frac{\theta' (x-\xinfcr)^t}{t!} = L(\xinfcr-) + \frac{\theta'}{t!} \sum_{j=0}^t \binom{t}{j}(-\xinfcr)^{t-j} x^j
\end{equation}
where $\theta'$ is the value of $\diff^t L/\diff x^t$ at some point within the interval $(\xinfcr-h,\xinfcr)$. The choice of $h$ assures that $(-1)^{t-1}\theta'$ is positive. Substituting \eqref{eq: Taylor} into \eqref{eq: riskas 1}, differentiating and making use of Lemma~\ref{lemma: der int} and \eqref{eq: gamma} gives
\begin{equation}
\label{eq: der princ 1}
\begin{split}
\frac{\diff\riskasinf_0}{\diff\Omega} & = \frac{L(\xinfcr-)}{\Omega(\nnum-1)!} \left[ -\left( \frac{\Omega}{\xinfcr} \right)^\nnum \exp( -\Omega/\xinfcr ) + \left( \frac{\Omega}{\xinfcr-h} \right)^\nnum \exp( -\Omega/(\xinfcr-h) ) \right] \\
& \quad + \frac{\theta'}{(\nnum-1)!\, t!} \sum_{j=0}^t \binom{t}{j}(-\xinfcr)^{t-j} \Omega^{j-1} \left[
-\left( \frac{\Omega}{\xinfcr} \right)^{\nnum-j} \exp( -\Omega/\xinfcr ) \right. \\
& \quad + \left( \frac{\Omega}{\xinfcr-h} \right)^{\nnum-j} \exp( -\Omega/(\xinfcr-h) ) \\
& \quad + \left. j \left( \Gamma\left(\nnum-j,\frac{\Omega}{\xinfcr}\right) - \Gamma\left(\nnum-j,\frac{\Omega}{\xinfcr-h}\right) \right) \right].
\end{split}
\end{equation}
The identity $\sum_{j=0}^t \binom{t}{j}(-1)^{t-j}=0$ implies that
\begin{equation}
\sum_{j=0}^t \binom{t}{j}(-\xinfcr)^{t-j} \Omega^{j-1} \left( \frac{\Omega}{\xinfcr} \right)^{\nnum-j}
= \xinfcr^{t-\nnum} \Omega^{\nnum-1} \sum_{j=0}^t \binom{t}{j}(-1)^{t-j} = 0,
\end{equation}
and thus \eqref{eq: der princ 1} simplifies to
\begin{equation}
\label{eq: der princ 1bis}
\begin{split}
\frac{\diff\riskasinf_0}{\diff\Omega} & = \frac{L(\xinfcr-) \Omega^{\nnum-1}}{(\nnum-1)!} \left( -\frac{\exp( -\Omega/\xinfcr)}{\xinfcr^\nnum} + \frac{\exp(-\Omega/(\xinfcr-h))}{(\xinfcr-h)^\nnum} \right) \\
& \quad + \frac{\Omega^{\nnum-1} \theta' \exp(-\Omega/(\xinfcr-h))}{(\nnum-1)!\, t!} \sum_{j=0}^t \binom{t}{j} \frac{(-\xinfcr)^{t-j}}{(\xinfcr-h)^{\nnum-j}} \\
& \quad + \frac{\theta'}{(\nnum-1)!\, t!} \sum_{j=0}^t \binom{t}{j}j(-\xinfcr)^{t-j} \Omega^{j-1} \left[ \Gamma\left(\nnum-j,\frac{\Omega}{\xinfcr}\right) - \Gamma\left(\nnum-j,\frac{\Omega}{\xinfcr-h}\right) \right].
\end{split}
\end{equation}
From Lemma~\ref{lemma: Gamma}, $\Omega^{j-1}\Gamma(\nnum-j,\Omega/\xinfcr)$ for $j \leq \nnum-1$ is given by
\begin{equation}
\label{eq: Omega Gamma a}
\Omega^{j-1}\Gamma\left(\nnum-j,\frac{\Omega}{\xinfcr}\right) = \exp( -\Omega/\xinfcr ) \sum_{k=j}^{\nnum-1} \frac{ \ff{(\nnum-j-1)}{\nnum-k-1} \Omega^{k-1}}{ \xinfcr^{k-j} },
\end{equation}
whereas for $j \geq \nnum$ and for any $w \in \mathbb N$
\begin{equation}
\label{eq: Omega Gamma b}
\begin{split}
\Omega^{j-1}\Gamma\left(\nnum-j,\frac{\Omega}{\xinfcr}\right) & = \exp( -\Omega/\xinfcr ) \sum_{k=\nnum-w}^{\nnum-1} \frac{ \ff{(\nnum-j-1)}{\nnum-k-1} \Omega^{k-1} }{ \xinfcr^{k-j} } \\
& \quad + O \left( \Omega^{\nnum-w-2} \exp( -\Omega/\xinfcr ) \right).
\end{split}
\end{equation}
Replacing $\xinfcr$ by $\xinfcr-h$ in \eqref{eq: Omega Gamma a} and \eqref{eq: Omega Gamma b} it is seen that
\begin{equation}
\label{eq: Omega Gamma desp}
\Omega^{j-1}\Gamma\left(\nnum-j,\frac{\Omega}{\xinfcr-h}\right) =
O \left( \Omega^{\nnum-2} \exp( -\Omega/(\xinfcr-h) ) \right).
\end{equation}
Setting $w = t$ in \eqref{eq: Omega Gamma b} and substituting \eqref{eq: Omega Gamma a}--\eqref{eq: Omega Gamma desp}
into \eqref{eq: der princ 1bis} yields
\begin{equation}
\label{eq: der princ 2}
\begin{split}
\frac{\diff\riskasinf_0}{\diff\Omega} & = -\frac{L(\xinfcr-) \Omega^{\nnum-1} \exp(-\Omega/\xinfcr)}{\xinfcr^{\nnum}(\nnum-1)!} \\
& \quad + \frac{\theta' \exp(-\Omega/\xinfcr)}{(\nnum-1)!\, t!} \left[ \sum_{j=0}^{\min\{t,\nnum-1\}} \binom{t}{j} j (-\xinfcr)^{t-j} \sum_{k=j}^{\nnum-1} \frac{\ff{(\nnum-j-1)}{\nnum-k-1} \Omega^{k-1}}{\xinfcr^{k-j}} \right. \\
& \quad + \left. \sum_{j=\nnum}^t \binom{t}{j}j(-\xinfcr)^{t-j} \sum_{k=\nnum-t}^{\nnum-1} \frac{\ff{(\nnum-j-1)}{\nnum-k-1} \Omega^{k-1}}{\xinfcr^{k-j}} \right] \\
& \quad + O \left( \Omega^{\nnum-t-2} \exp( -\Omega/\xinfcr) \right)
\end{split}
\end{equation}
(the term $O( \Omega^{\nnum-t-2} \exp(-\Omega/\xinfcr) )$ could be substituted by a lower-order term if $t<\nnum$, but this is unnecessary for the proof). Since $\ff{(\nnum-j-1)}{\nnum-k-1}=0$ for $k < j < \nnum$, the summation range of the first sum over $k$ in \eqref{eq: der princ 2} can be extended from $k=j,\ldots,\nnum-1$ to $k=\min\{0,\nnum-t\},\ldots,\nnum-1$. On the other hand, the second sum over $j$ is empty if $t < \nnum$. Thus the second sum over $k$ only appears if $t \geq \nnum$, and in this case $\min\{0,\nnum-t\} = \nnum-t$. Therefore the lower limit in the latter sum can also be expressed as $k=\min\{0,\nnum-t\}$. With these changes, \eqref{eq: der princ 2} is rewritten as
\begin{equation}
\label{eq: der princ 3}
\begin{split}
\frac{\diff\riskasinf_0}{\diff\Omega} & = -\frac{L(\xinfcr-) \Omega^{\nnum-1} \exp(-\Omega/\xinfcr)}{\xinfcr^{\nnum}(\nnum-1)!} + \frac{\xinfcr^{t-1} \theta' \exp(-\Omega/\xinfcr)}{(\nnum-1)!\, t!} \\
& \quad \cdot \sum_{k=\min\{0,\nnum-t\}}^{\nnum-1} \left(\frac{\Omega}{\xinfcr}\right)^{k-1} \sum_{j=0}^t \binom{t}{j}j \ff{(\nnum-j-1)}{\nnum-k-1} (-1)^{t-j} \\
& \quad + O \left( \Omega^{\nnum-t-2} \exp(-\Omega/\xinfcr) \right).
\end{split}
\end{equation}
From Lemma~\ref{lemma: der binom partic}, the inner sum in \eqref{eq: der princ 3} equals $0$ for $k=\nnum-t+1,\nnum-t+2,\ldots,\nnum-1$ and $(-1)^{t-1} t!$ for $k=\nnum-t$. If $t < \nnum$, the terms with index $k=0,1,\ldots,\nnum-t-1$ are $O ( \Omega^{\nnum-t-2} \exp(-\Omega/\xinfcr) )$. Therefore
\begin{equation}
\label{eq: der princ 4}
\begin{split}
\frac{\diff\riskasinf_0}{\diff\Omega} & =
-\frac{L(\xinfcr-) \Omega^{\nnum-1} \exp(-\Omega/\xinfcr)}{\xinfcr^{\nnum}(\nnum-1)!}
+ \frac{(-1)^{t-1} \xinfcr^{2t-\nnum} \theta' \Omega^{\nnum-t-1} \exp(-\Omega/\xinfcr)}{(\nnum-1)!} \\
& \quad + O \left( \Omega^{\nnum-t-2} \exp(-\Omega/\xinfcr) \right).
\end{split}
\end{equation}

Using \eqref{eq: der riskas 1''''}, \eqref{eq: der riskas 1'''}, \eqref{eq: der riskas 1'}, \eqref{eq: der riskas 1''} and \eqref{eq: der princ 4}, and considering that $L(\xinfcr-)=L(\xinfcr+)$,
\begin{equation}
\label{eq: der Omega infty grado mayor}
\frac{\diff\riskas}{\diff\Omega} \geq \frac{(-1)^{t-1} \xinfcr^{2t-\nnum} \theta' \Omega^{\nnum-t-1} \exp(-\Omega/\xinfcr)}{(\nnum-1)!}
+ O \left( \Omega^{\nnum-t-2} \exp(-\Omega/\xinfcr) \right).
\end{equation}
Since $(-1)^{t-1} \theta' > 0$, this implies that
\begin{equation}
\label{eq: der Omega infty grado mayor 2}
\lim_{\Omega \rightarrow \infty} \left( \frac{\exp(\Omega/\xinfcr)}{\Omega^{\nnum-t-1}} \frac{\diff\riskas}{\diff\Omega} \right) \geq \frac{(-1)^{t-1} \xinfcr^{2t-\nnum} \theta'}{(\nnum-1)!} > 0.
\end{equation}
As a consequence of \eqref{eq: lim der Omega infty grado 0} and \eqref{eq: der Omega infty grado mayor 2}, in either case~\ref{assum item item: i} or \ref{assum item item: ii} of property~\ref{assum item: bien der} in Assumption~\ref{assum: L crec}, there exists $\Omega_0'$ such that $\diff\riskas/\diff\Omega > 0$ for $\Omega \geq \Omega_0'$.
\end{proof}

\begin{proof}[Proof of Theorem~\ref{teo: min resp Omega}]
From Lemmas~\ref{lemma: min resp Omega izq} and \ref{lemma: min resp Omega der}, there exist $\Omega_0$, $\Omega_0'$ such that, denoting by ${\riskas|}_\Omega$ the value of $\riskas$ corresponding to a given $\Omega$,
\begin{align}
\label{eq: max en int cerr 1}
{\riskas|}_\Omega & > {\riskas|}_{\Omega_0} \quad \text{for } \Omega < \Omega_0, \\
\label{eq: max en int cerr 2}
{\riskas|}_\Omega & > {\riskas|}_{\Omega_0'} \quad \text{for } \Omega > \Omega_0'.
\end{align}
Proposition~\ref{prop: riskas C1 Omega} implies that $\riskas$ is a continuous function of $\Omega$. Therefore, this function restricted to the interval $[\Omega_0, \Omega_0']$ has an absolute maximum \citep[theorem 4.28]{Apostol74}. Because of \eqref{eq: max en int cerr 1} and \eqref{eq: max en int cerr 2}, this is the absolute maximum of $\riskas$ over $\mathbb R^+$.
\end{proof}


\begin{lemma}
\label{lemma: sop fin cont}
Under the hypotheses of Theorem~\ref{teo: opt}, given $\csopfin \in \mathbb R^+$, $\riskasfs$ as defined by \eqref{eq: riskas sop fin} is a continuous function of $\Omega \in \mathbb R^+$.
\end{lemma}

\begin{proof}
From
Assumptions~\ref{assum: L var acot} and \ref{assum: L disc fin}, $L$ is continuous except possibly at a finite number of points, where it can only have removable discontinuities or jumps. Since removable discontinuities do not have any effect on the integral in \eqref{eq: riskas sop fin}, they can be disregarded. Thus in the following it is assumed that $L$ only has jump discontinuities. Let $D$ be the number of discontinuity points, located at $x_1 < x_2 < \cdots < x_D$. The function $L$ can be decomposed as the sum of $L_\mathrm c$ and $L_\mathrm d$, where $L_\mathrm c$ is continuous and $L_\mathrm d$ is piecewise constant with jumps at $x_1,\ldots,x_D$. Accordingly, $\riskasfs = \riskasfs_\mathrm c + \riskasfs_\mathrm d$, where $\riskasfs_\mathrm c$ and $\riskasfs_\mathrm d$ are given as in \eqref{eq: riskas sop fin} with $L$ replaced by $L_\mathrm c$ and $L_\mathrm d$ respectively.

For any $\Omega' \neq \Omega$, let $\riskasfs'$ denote the right-hand side of \eqref{eq: riskas sop fin} with $\Omega$ replaced by $\Omega'$, and let $\riskasfs'_\mathrm c$ and $\riskasfs'_\mathrm d$ be defined similarly. For $\epsilon>0$ arbitrary, it is necessary to find $\delta>0$ such that $|\riskasfs'-\riskasfs|<\epsilon$ for $|\Omega'-\Omega|<\delta$. Consider an arbitrary $\delta_0 \in (0,\Omega)$. Since $L_\mathrm c$ is continuous, by the Heine-Cantor theorem \citep[theorem 4.47]{Apostol74} it is uniformly continuous on the interval $[(\Omega-\delta_0)/(\nnum \csopfin), (\Omega+\delta_0) \csopfin/\nnum]$. This interval contains the values $\Omega/\nu$ and $\Omega'/\nu$ for $|\Omega'-\Omega|<\delta_0$, $\nu \in [\nnum/\csopfin, \nnum\csopfin]$. By virtue of this, defining $\epsilon_\mathrm c = \epsilon/(2\nnum(\csopfin-1/\csopfin))$, let $\delta_\mathrm c<\delta_0$ be chosen such that $|L_\mathrm c(\Omega'/\nu)-L_\mathrm c(\Omega/\nu)| < \epsilon_\mathrm c$ for all $|\Omega'-\Omega|<\delta_\mathrm c$, $\nu \in [\nnum/\csopfin, \nnum\csopfin]$. Taking into account Lemma~\ref{lemma: phi}, it follows that
\begin{equation}
\label{eq: L cont epsilon}
|\riskasfs'_\mathrm c-\riskasfs_\mathrm c|
\leq \int_{\nnum/\csopfin}^{\nnum \csopfin} |L( \Omega/\nu )-L( \Omega'/\nu )| \,\diff \nu
< \nnum \left(\csopfin-\frac 1 \csopfin \right) \epsilon_\mathrm c = \frac \epsilon 2 \quad \text{for } |\Omega'-\Omega|<\delta_\mathrm c.
\end{equation}

By construction, there exists an upper bound $M_\mathrm d$ on $|L_\mathrm d(x)|$, $x \in \mathbb R^+$. Since $L_\mathrm d(\Omega/\nu)$, considered as a function of $\nu$, has jumps at $\Omega/x_1,\ldots,\Omega/x_D$, associated with each discontinuity point $\Omega/x_i$ there is an interval of values of $\nu$ for which $L_\mathrm d(\Omega'/\nu) \neq L_\mathrm d(\Omega/\nu)$. The width of this interval is $|\Omega'-\Omega|/x_i \leq |\Omega'-\Omega|/x_1$, and $|L_\mathrm d(\Omega'/\nu) - L_\mathrm d(\Omega/\nu)| \leq 2 M_\mathrm d$ for $\nu$ within this interval. There are at most $D$ such intervals contained in $[\nnum/\csopfin,\nnum\csopfin]$, and for any value of $\nu$ not belonging to any of these intervals it holds that $L_\mathrm d(\Omega'/\nu) = L_\mathrm d(\Omega/\nu)$. Using Lemma~\ref{lemma: phi} again, it is seen that $|\riskasfs'_\mathrm d-\riskasfs_\mathrm d| \leq 2 D M_\mathrm d |\Omega'-\Omega|/x_1$. Thus there exists $\delta_\mathrm d$ such that
\begin{equation}
\label{eq: L disc epsilon}
|\riskasfs'_\mathrm d-\riskasfs_\mathrm d| < \frac \epsilon 2 \quad \text{for } |\Omega'-\Omega|<\delta_\mathrm d.
\end{equation}
Taking $\delta = \min\{\delta_\mathrm c,\delta_\mathrm d\}$, it follows from \eqref{eq: L cont epsilon} and \eqref{eq: L disc epsilon} that
\begin{equation}
|\riskasfs'-\riskasfs| \leq |\riskasfs'_\mathrm c-\riskasfs_\mathrm c| + |\riskasfs'_\mathrm d-\riskasfs_\mathrm d| < \frac \epsilon 2 + \frac \epsilon 2 = \epsilon \quad \text{for } |\Omega'-\Omega|<\delta,
\end{equation}
which shows that $\riskasfs$ is a continuous function of $\Omega$.
\end{proof}

\begin{lemma}
\label{lemma: sop fin lim limsup}
Under the hypotheses of Theorem~\ref{teo: opt}, and with $\riskasfs$ defined by \eqref{eq: riskas sop fin},
\begin{equation}
\lim_{\csopfin \rightarrow \infty} \limsup_{\Omega \rightarrow 0} \frac{\riskas-\riskasfs}{\riskasfs}
= \lim_{\csopfin \rightarrow \infty} \limsup_{\Omega \rightarrow \infty} \frac{\riskas-\riskasfs}{\riskasfs}
= 0.
\end{equation}
\end{lemma}

\begin{proof}
According to property~\ref{assum item: cota Theta, cero} in Assumption~\ref{assum: L cota Theta}, there exist $\Kcero < \nnum$ and $\mcero, \Mcero, \xcero \in \mathbb R^+$ such that $\mcero x^\Kcero < L(x) < \Mcero x^\Kcero$ for $x < \xcero$, that is,
\begin{equation}
\label{eq: cota Theta, cero}
\mcero (\Omega/\nu)^\Kcero  < L(\Omega/\nu) < \Mcero (\Omega/\nu)^\Kcero \quad \text{for } \nu > \Omega/\xcero.
\end{equation}
Similarly, property~\ref{assum item: cota Theta, inf} implies that there exist $\Kinf < \nnum$; $\minf, \Minf \in \mathbb R^+$; and $\xinf > \xcero$ such that
\begin{equation}
\label{eq: cota Theta, inf}
\minf (\Omega/\nu)^\Kinf < L(\Omega/\nu) < \Minf (\Omega/\nu)^\Kinf \quad \text{for } \nu < \Omega/\xinf.
\end{equation}
From Assumption~\ref{assum: L var acot}, $L$ is of bounded variation on $[\xcero, \xinf]$, and thus there exists $M$ such that $L(x) \leq M$ for $x \in [\xcero, \xinf]$, that is,
\begin{equation}
\label{eq: cota Theta, medio}
L(\Omega/\nu) \leq M \quad \text{for } \Omega/\xinf \leq \nu \leq \Omega/\xcero.
\end{equation}

The case $\Omega \rightarrow 0$ is analyzed first. Given $\csopfin \in \mathbb R^+$, it will be assumed that $\Omega < \nnum \xcero/\csopfin$. Under this assumption, any $\nu$ within the integration interval in \eqref{eq: riskas sop fin} exceeds $\Omega/\xcero$. Thus, applying \eqref{eq: cota Theta, cero},
\begin{equation}
\riskasfs > \mcero \Omega^\Kcero \int_{\nnum/\csopfin}^{\nnum \csopfin} \frac{\nu^{\nnum-\Kcero-1} \exp(-\nu)}{(\nnum-1)!} \,\diff \nu
=  \frac{\mcero \Omega^\Kcero ( \Gamma(\nnum-\Kcero,\nnum/\csopfin) - \Gamma(\nnum-\Kcero,\nnum \csopfin) )}{(\nnum-1)!}.
\end{equation}

The difference $\riskas-\riskasfs$ can be expressed as\footnote{Note that this decomposition, and the one to be used for $\Omega \rightarrow \infty$, are different from those used in the proofs of Lemmas~\ref{lemma: min resp Omega izq} and \ref{lemma: min resp Omega der} respectively, although the same notation is used for simplicity.}
$\riskascero_1 + \riskascero_2 + \riskascero_3 + \riskascero_4$, where each term is an integral as in \eqref{eq: riskas sop fin} with the integration interval respectively given as $(0,\Omega/\xinf)$, $(\Omega/\xinf, \Omega/\xcero)$, $(\Omega/\xcero, \nnum/\csopfin)$ and $(\nnum \csopfin, \infty)$. In the first case, \eqref{eq: cota Theta, inf} implies that
\begin{equation}
\riskascero_1 < \frac{\Minf \Omega^\Kinf \gamma(\nnum-\Kinf,\Omega/\xinf)}{(\nnum-1)!},
\end{equation}
and thus
\begin{equation}
\label{eq: lemma: sop fin 1}
\frac{\riskascero_1}{\riskasfs} < \frac{\Minf \Omega^{\Kinf-\Kcero} \gamma(\nnum-\Kinf,\Omega/\xinf)}{\mcero (\Gamma(\nnum-\Kcero,\nnum/\csopfin) - \Gamma(\nnum-\Kcero,\nnum \csopfin))}.
\end{equation}
Using the equality \eqref{eq: lim gamma 0} from Lemma~\ref{lemma: lim gamma}, and taking into account that $\Kcero, \Kinf < \nnum$ by Assumption~\ref{assum: L cota Theta}, it is seen that the right-hand side of \eqref{eq: lemma: sop fin 1} tends to $0$ as $\Omega \rightarrow 0$. Since $\riskascero_1$ and $\riskasfs$ are both positive, this implies that
\begin{equation}
\label{eq: lemma: sop fin 2}
\lim_{\Omega \rightarrow 0} \frac{\riskascero_1}{\riskasfs} = 0.
\end{equation}
As for the term $\riskascero_2$, using \eqref{eq: cota Theta, medio},
\begin{equation}
\riskascero_2 \leq \frac{M ( \gamma(\nnum,\Omega/\xcero) - \gamma(\nnum,\Omega/\xinf) )}{(\nnum-1)!} < \frac{M \gamma(\nnum,\Omega/\xcero)}{(\nnum-1)!},
\end{equation}
and thus
\begin{equation}
\frac{\riskascero_2}{\riskasfs} < \frac{ M \Omega^{-\Kcero} \gamma(\nnum,\Omega/\xcero) }
{ \mcero ( \Gamma(\nnum-\Kcero,\nnum/\csopfin) - \Gamma(\nnum-\Kcero,\nnum \csopfin) ) }.
\end{equation}
Using \eqref{eq: lim gamma 0} again, and taking into account that $\Kcero < \nnum$, it stems that
\begin{equation}
\label{eq: lemma: sop fin 3}
\lim_{\Omega \rightarrow 0} \frac{\riskascero_2}{\riskasfs} = 0.
\end{equation}
Regarding the third term, \eqref{eq: cota Theta, cero} holds for all $\nu$ within the integration interval, and thus
\begin{equation}
\riskascero_3 < \frac{\Mcero \Omega^\Kcero (\gamma(\nnum-\Kcero,\nnum/\csopfin) - \gamma(\nnum-\Kcero,\Omega/\xcero) )}{(\nnum-1)!} < \frac{\Mcero \Omega^\Kcero \gamma(\nnum-\Kcero,\nnum/\csopfin)}{(\nnum-1)!}.
\end{equation}
Therefore
\begin{equation}
\label{eq: lemma: sop fin 4}
\frac{\riskascero_3}{\riskasfs} < \frac{ \Mcero \gamma(\nnum-\Kcero,\nnum/\csopfin)}
{ \mcero ( \Gamma(\nnum-\Kcero,\nnum/\csopfin) - \Gamma(\nnum-\Kcero,\nnum \csopfin) ) }.
\end{equation}
Similarly, the fourth term satisfies
\begin{equation}
\riskascero_4 < \frac{\Mcero \Omega^\Kcero \Gamma(\nnum-\Kcero,\nnum \csopfin)}{(\nnum-1)!},
\end{equation}
and therefore
\begin{equation}
\label{eq: lemma: sop fin 5}
\frac{\riskascero_4}{\riskasfs} < \frac{ \Mcero \Gamma(\nnum-\Kcero,\nnum \csopfin)}
{ \mcero ( \Gamma(\nnum-\Kcero,\nnum/\csopfin) - \Gamma(\nnum-\Kcero,\nnum \csopfin) ) }.
\end{equation}

From \eqref{eq: lemma: sop fin 2}, \eqref{eq: lemma: sop fin 3}, \eqref{eq: lemma: sop fin 4} and \eqref{eq: lemma: sop fin 5} it follows that
\begin{equation}
\label{eq: lemma sop fin res 0}
\limsup_{\Omega \rightarrow 0} \frac{\riskas-\riskasfs}{\riskasfs} \leq
\frac{ \Mcero ( \gamma(\nnum-\Kcero,\nnum/\csopfin) + \Gamma(\nnum-\Kcero,\nnum \csopfin))}
{ \mcero ( \Gamma(\nnum-\Kcero,\nnum/\csopfin) - \Gamma(\nnum-\Kcero,\nnum \csopfin) ) }.
\end{equation}
The right-hand side of \eqref{eq: lemma sop fin res 0} is seen to converge to $0$ as $\csopfin \rightarrow \infty$, and thus so does the left-hand side. This establishes the first part of the result.

The analysis for $\Omega \rightarrow \infty$ is similar. Given $\csopfin \in \mathbb R^+$, it is assumed that $\Omega > \nnum \xinf \csopfin$. The difference $\riskas-\riskasfs$ is expressed as $\riskasinf_1 + \riskasinf_2 + \riskasinf_3 + \riskasinf_4$, where each term is an integral as in \eqref{eq: riskas sop fin} with integration intervals respectively given as $(0,\nnum/\csopfin)$, $(\nnum\csopfin, \Omega/\xinf)$, $(\Omega/\xinf, \Omega/\xcero)$ and $(\Omega/\xcero, \infty)$. Arguments analogous to those used for $\Omega \rightarrow 0$ establish that
\begin{equation}
\label{eq: lemma sop fin res inf}
\limsup_{\Omega \rightarrow \infty} \frac{\riskas-\riskasfs}{\riskasfs} \leq \frac{ \Minf ( \gamma(\nnum-\Kinf,\nnum/\csopfin) + \Gamma(\nnum-\Kinf,\nnum \csopfin) ) }
{ \minf ( \Gamma(\nnum-\Kinf,\nnum/\csopfin) - \Gamma(\nnum-\Kinf,\nnum \csopfin) ) }.
\end{equation}
The right-hand side of \eqref{eq: lemma sop fin res inf} is seen to converge to $0$ as $\csopfin \rightarrow \infty$, and thus so does the left-hand side. This establishes the second part of the result.
\end{proof}

\begin{lemma}
\label{lemma: sop fin conv unif}
Under the hypotheses of Theorem~\ref{teo: opt}, considering $\riskasfs$ and $\riskas$ as functions of $\Omega \in \mathbb R^+$, $\riskasfs/\riskas \rightarrow 1$ uniformly on $\mathbb R^+$ as $\csopfin \rightarrow \infty$.
\end{lemma}

\begin{proof}
The result is equivalent to the statement that for any $\epsilon>0$ there exists $\csopfin_0$ such that $|\riskas/\riskasfs-1|<\epsilon$ for all $\Omega \in \mathbb R^+$ and for all $\csopfin > \csopfin_0$. Consider $\epsilon>0$ arbitrary. Let $R(\csopfin)$ and $R'(\csopfin)$ respectively denote $\limsup_{\Omega \rightarrow 0} (\riskas-\riskasfs)/\riskasfs$ and $\limsup_{\Omega \rightarrow \infty} (\riskas-\riskasfs)/\riskasfs$. Since $L$ is a non-negative function, from \eqref{eq: riskas sop fin} it is seen that $\riskasfs$ is a non-negative, non-decreasing function of $\csopfin$ for any $\Omega$. By Lemma~\ref{lemma: sop fin lim limsup}, $R(\csopfin)$ and $R'(\csopfin)$ tend to $0$ as $\csopfin \rightarrow \infty$, and thus there exists $\csopfin_1$ such that $R(\csopfin_1), R'(\csopfin_1) \leq \epsilon/2$. By definition of $R(\csopfin)$, there exists $\Omega_0$ such that the following inequality holds (note that the left-hand side is a function of $\sigma$ and $\Omega$):
\begin{equation}
\label{eq: cota R pre}
\frac{\riskas-\riskasfs}{\riskasfs} < R(\csopfin_1)+ \frac \epsilon 2 \leq \epsilon \quad \text{for } \Omega<\Omega_0,\ \sigma=\sigma_1.
\end{equation}
The non-decreasing character of $\riskasfs$ with $\sigma$ implies that \eqref{eq: cota R pre} also holds for $\sigma > \sigma_1$, that is,
\begin{equation}
\label{eq: cota R}
\frac{\riskas-\riskasfs}{\riskasfs} < \epsilon \quad \text{for } \Omega < \Omega_0, \ \csopfin \geq \csopfin_1.
\end{equation}
Analogously, there exists $\Omega_0'>\Omega_0$ such that
\begin{equation}
\label{eq: cota R'}
\frac{\riskas-\riskasfs}{\riskasfs} < \epsilon \quad \text{for } \Omega > \Omega_0', \ \csopfin \geq \csopfin_1.
\end{equation}

According to Lemma~\ref{lemma: sop fin cont}, for $\csopfin$ fixed, $\riskasfs$ is a continuous function of $\Omega \in [\Omega_0,\Omega_0']$, and therefore it has an absolute minimum on that interval, which will be denoted as $S_1(\csopfin)$. The non-negative and non-decreasing character of $\riskasfs$ with $\sigma$ implies that $S_1$ is also a non-negative, non-decreasing function. In addition, $S_1(\csopfin) > 0$ for all $\csopfin$ greater than a certain value $\csopfin_2$. This can be seen as follows. By Assumption~\ref{assum: L cota Theta}, $L(x)$ is non-zero for all $x$ outside a bounded interval. If $\csopfin$ is sufficiently large, i.e.~greater than a certain $\csopfin_2$, for any $\Omega \in [\Omega_0,\Omega_0']$ the integration interval in \eqref{eq: riskas sop fin} contains a subinterval where $L$ is non-zero, which gives $\riskasfs>0$. Thus $S_1(\csopfin) > 0$ for $\csopfin > \csopfin_2$.

By arguments similar to those in the above paragraph, $\riskas-\riskasfs$, considered as a function of $\Omega$, has an absolute maximum on $[\Omega_0,\Omega_0']$; and this maximum, denoted as $S_2(\csopfin)$, tends to $0$ as $\csopfin \rightarrow \infty$. Therefore, defining $S(\csopfin)=S_2(\csopfin)/S_1(\csopfin)$ for $\csopfin > \csopfin_2$,
\begin{equation}
\label{eq: cota S pre}
(\riskas-\riskasfs)/\riskasfs \leq S(\csopfin) \quad \text{for }\Omega \in [\Omega_0,\Omega_0'],\ \csopfin > \csopfin_2;
\end{equation}
and $S(\csopfin) \rightarrow 0$ as $\csopfin \rightarrow \infty$. Thus, for the considered $\epsilon$, there exists $\csopfin_3 \geq \csopfin_2$ such that $S(\csopfin) < \epsilon$ for $\csopfin \geq \csopfin_3$. Combined with \eqref{eq: cota S pre}, this gives
\begin{equation}
\label{eq: cota S}
(\riskas-\riskasfs)/{\riskasfs} < \epsilon \quad \text{for } \Omega \in [\Omega_0,\Omega_0'], \ \csopfin \geq \csopfin_3.
\end{equation}
From \eqref{eq: cota R}, \eqref{eq: cota R'} and \eqref{eq: cota S}, choosing $\csopfin_0 = \max\{\csopfin_1,\csopfin_3\}$ is sufficient to satisfy $|\riskas/\riskasfs-1|<\epsilon$ for $\Omega \in \mathbb R^+$, $\csopfin > \csopfin_0$. This completes the proof.
\end{proof}

\begin{proof}[Proof of Theorem~\ref{teo: opt}]
The result will be proved by contradiction. Assume that there exists a possibly randomized estimator $\hatvap$ with $\limsup_{p \rightarrow 0} \risk(p) < \risk^*$. This implies that there exist $\theta < 1$ and a probability $p_\theta$ such that the estimator has
\begin{equation}
\label{eq: hip risk, teo opt}
\risk(p) < \theta \risk^* \quad \text{for all } p < p_\theta.
\end{equation}
For $\nden = \nnum, \nnum+1, \ldots$, let $\Pi_\nden$ denote the distribution function of $\hatvap$ conditioned on $\vanden = \nden$.

By Lemma~\ref{lemma: sop fin conv unif}, let $\csopfin$ be selected such that
\begin{equation}
\label{eq: teo opt 1}
\int_{\nnum/\csopfin}^{\nnum \csopfin} \phi(\nu) L( \Omega/\nu ) \,\diff \nu >
\sqrt[3]{\theta} \int_0^\infty \phi(\nu) L( \Omega/\nu ) \,\diff \nu \quad \text{for all } \Omega \in \mathbb R^+.
\end{equation}
In particular, this implies that
\begin{equation}
\label{eq: teo opt 1 bis}
\int_{\nnum/\csopfin}^{\nnum \csopfin} \phi(\nu) L( \Omega/\nu ) \,\diff \nu >
\sqrt[3]{\theta} \risk^* \quad \text{for all } \Omega \in \mathbb R^+.
\end{equation}

Given $\nu_1, \nu_2$ with $\nu_2 > \nu_1 > 0$, according to Lemma~\ref{lemma: Phi conv unif}, $\Phi(p,\nu) \rightarrow \phi(\nu)$ uniformly on $[\nu_1,\nu_2]$ as $p \rightarrow 0$. By virtue of this, let $p_1 < p_\theta$ be such that
\begin{equation}
\label{eq: teo opt 1 ter}
|\Phi(p,\nu)-\phi(\nu)| < (1 - \sqrt[3]{\theta}) \phi(\nu) \quad \text{for all } p < p_1,\ \nu \in [\nnum/\csopfin, \nnum \csopfin].
\end{equation}
Let $\ndenini = \lceil \nnum \csopfin/p_1 \rceil$. Taking into account that $\lim_{w \rightarrow \infty} ( \sum_{\nden=1}^w 1/\nden - \log w ) = \gamma$, where $\gamma$ is the Euler-Mascheroni constant \citep[equation 6.1.3]{Abramowitz70}, it is easy to see that
\begin{equation}
\lim_{p \rightarrow 0} \left( \sum_{\nden=\ndenini}^{\lfloor \nnum/(\csopfin p) \rfloor} \frac{1}{\nden} - \log\frac{p_1}{p} \right) = \gamma + \log\frac{\nnum}{\csopfin p_1} - \sum_{\nden=1}^{\ndenini-1} \frac{1}{\nden}.
\end{equation}
This implies that there exist $\lambda>0$ and $p'_0$ such that
\begin{equation}
\label{eq: teo opt 2}
\sum_{\nden=\ndenini}^{\lfloor \nnum/(\csopfin p) \rfloor} \frac{1}{\nden} - \log\frac{p_1}{p} > -\lambda \quad \text{for all }p \leq p'_0.
\end{equation}
Let $\lambda$ and $p'_0$ be chosen such that \eqref{eq: teo opt 2} holds, and let $p''_0$ be defined by the equation
\begin{equation}
\label{eq: teo opt 2 bis}
\log\frac{p_1}{p''_0} = \frac{\lambda}{1-\sqrt[3]{\theta}}.
\end{equation}
Since $\lambda>0$ and $\theta<1$, it follows that $p''_0 < p_1$.

Let $p_0 = \min\{p'_0, p''_0\}$. For a given $\nden$, the measure associated with the distribution function $\Pi_\nden$ is obviously finite, and thus sigma-finite. This implies \citep[theorem 18.3]{Billingsley95} that for each $\nden$ the integral in \eqref{eq: risk gen}, considered as a function of $p$, is measurable with respect to Lebesgue measure. In addition, since $p_1 < p_\theta$, it stems from \eqref{eq: hip risk, teo opt} that the series in \eqref{eq: risk gen} converges for $p \leq p_1$. This assures \citep[theorem 13.4(ii)]{Billingsley95} that $\risk(p)$ restricted to $p \leq p_1$ is measurable. Therefore, the integral
\begin{equation}
\label{eq: teo opt 3}
X = \int_{p_0}^{p_1} \frac{\risk(p)}{p} \,\diff p
\end{equation}
exists in the Lebesgue sense, and according to \eqref{eq: hip risk, teo opt} it satisfies
\begin{equation}
\label{eq: teo opt 3 bis}
X < \theta \risk^* \int_{p_0}^{p_1} \frac{\diff p}{p} = \theta \risk^* \log\frac{p_1}{p_0}.
\end{equation}
Substituting \eqref{eq: risk gen} into \eqref{eq: teo opt 3},
\begin{equation}
\label{eq: teo opt 3 ter}
X = \int_{p_0}^{p_1} \frac 1 p \sum_{\nden=\nnum}^\infty f(\nden) \left( \int_0^\infty L(y/p) \,\diff \Pi_\nden(y) \right) \,\diff p.
\end{equation}
Defining $\ndenfin = \lfloor \nnum/(\csopfin p_0) \rfloor$, it is clear from \eqref{eq: teo opt 3 ter} that
\begin{equation}
\label{eq: teo opt 4}
X > \sum_{\nden=\ndenini}^{\ndenfin} \int_{p_0}^{p_1} \left( \int_0^\infty \frac{f(\nden) L(y/p)}{p} \,\diff \Pi_\nden(y) \right) \,\diff p.
\end{equation}
Since both measures in \eqref{eq: teo opt 4} are sigma-finite, and both the inner and outer integrals are finite, the order of integration can be reversed \citep[theorem~18.3]{Billingsley95}, which gives
\begin{equation}
\label{eq: teo opt 5}
X > \sum_{\nden=\ndenini}^{\ndenfin} \int_0^\infty \left( \int_{p_0}^{p_1} \frac{f(\nden) L(y/p)}{p} \,\diff p \right) \diff \Pi_\nden(y) .
\end{equation}
Making the change of variable $\nu = \nden p$ in the inner integral and taking into account that $f(\nden)/p = \Phi(p,\nden p)$, \eqref{eq: teo opt 5} becomes
\begin{equation}
\label{eq: teo opt 6}
X > \sum_{\nden=\ndenini}^{\ndenfin} \frac{1}{\nden} \int_0^\infty \left( \int_{\nden p_0}^{\nden p_1} \Phi(\nu/\nden,\nu) L(\nden y/\nu) \,\diff \nu \right) \diff \Pi_\nden(y) .
\end{equation}
For $\ndenini \leq \nden \leq \ndenfin$ it holds that $\nden p_0 \leq \nnum/\csopfin$ and $\nden p_1 \geq \nnum \csopfin$. Therefore
\begin{equation}
\label{eq: teo opt 7}
X > \sum_{\nden=\ndenini}^{\ndenfin} \frac{1}{\nden} \int_0^\infty \left( \int_{\nnum/\csopfin}^{\nnum \csopfin} \Phi(\nu/\nden,\nu) L(\nden y/\nu) \,\diff \nu \right) \diff \Pi_\nden(y) .
\end{equation}
For $\nu \in [\nnum/\csopfin, \nnum\csopfin]$ and $\ndenini \leq \nden \leq \ndenfin$ it holds that $\nu/\nden < p_1$. Thus \eqref{eq: teo opt 1 ter} gives $\Phi(\nu/\nden,\nu) > \sqrt[3]{\theta} \phi(\nu)$. Substituting into \eqref{eq: teo opt 7},
\begin{equation}
\label{eq: teo opt 8}
X > \sqrt[3]{\theta} \sum_{\nden=\ndenini}^{\ndenfin} \frac{1}{\nden} \int_0^\infty \left( \int_{\nnum/\csopfin}^{\nnum \csopfin} \phi(\nu) L(\nden y/\nu) \,\diff \nu \right) \diff \Pi_\nden(y).
\end{equation}
From \eqref{eq: teo opt 1 bis}, the inner integral in \eqref{eq: teo opt 8} exceeds $\sqrt[3]{\theta} \risk^*$, and thus
\begin{equation}
\label{eq: teo opt 9}
X > \theta^{2/3} \risk^* \sum_{\nden=\ndenini}^{\ndenfin} \frac{1}{\nden}.
\end{equation}
Since $p_0 \leq p'_0$ and $p_0 \leq p''_0 < p_1$, \eqref{eq: teo opt 2} and \eqref{eq: teo opt 2 bis} give
\begin{equation}
\label{eq: teo opt 10}
\sum_{\nden=\ndenini}^{\ndenfin} \frac{1}{\nden} > -\lambda + \log \frac{p_1}{p_0}
\geq \log \frac{p_1}{p_0} \left(1 - \frac{\lambda}{\log(p_1/p''_0)} \right) = \sqrt[3]{\theta} \log \frac{p_1}{p_0}.
\end{equation}
Substituting into \eqref{eq: teo opt 9},
\begin{equation}
\label{eq: teo opt 11}
X > \theta \risk^* \log \frac{p_1}{p_0},
\end{equation}
in contradiction with \eqref{eq: teo opt 3 bis}. This establishes the result.
\end{proof}

\begin{proof}[Proof of Proposition~\ref{prop: as nonas}]
The proof is analogous to that of \citet[proposition~1]{Mendo10}.
\end{proof}

\begin{proof}[Proof of Proposition~\ref{prop: MSE gar}]
For the considered estimator,
\begin{equation}
\label{eq: risk MSE}
\frac{\E[(\hatvap -p)^2]}{p^2} = \frac{(\nnum-2)^2}{p^2} \E\left[\frac{1}{(\vanden-1)^2}\right] - \frac{2(\nnum-2)}{p} \E\left[\frac{1}{\vanden-1}\right] + 1.
\end{equation}
The equality
\begin{equation}
\label{eq: E 1 vanden 1}
\E\left[\frac{1}{\vanden-1}\right] = \frac{p}{\nnum-1}
\end{equation}
directly stems from the fact that \eqref{eq: hatvap nnum 1 vanden 1} is unbiased. On the other hand, according to \citet{Mikulski76}, for $p \in (0,1)$
\begin{equation}
\label{eq: E 1 vanden 1 cuad}
\Var\left[\frac{\nnum-1}{\vanden-1}\right] \leq \frac{p^2(1-p)}{\nnum-2} < \frac{p^2}{\nnum-2}.
\end{equation}
From \eqref{eq: E 1 vanden 1} and \eqref{eq: E 1 vanden 1 cuad},
\begin{equation}
\label{eq: E 1 vanden 1 cuad 2}
\E\left[\frac{1}{(\vanden-1)^2}\right]
= \left(\E\left[\frac{1}{\vanden-1}\right]\right)^2 + \Var\left[\frac{1}{\vanden-1}\right]
< \frac{p^2}{(\nnum-1)(\nnum-2)}.
\end{equation}
Substituting \eqref{eq: E 1 vanden 1} and \eqref{eq: E 1 vanden 1 cuad 2} into \eqref{eq: risk MSE}, the desired result \eqref{eq: MSE gar} is obtained.
\end{proof}


\begin{lemma}
\label{lemma: conf gen gar}
Given $\nnum \geq 3$ and $\Omega \in \mathbb R^+$, considering the loss function \eqref{eq: L conf gen} with $A_1=0$, $A_2 > 0$, if $\fdos \geq (\nnum+\sqrt\nnum+1)/\Omega$ the risk of the estimator \eqref{eq: hatvap Omega vanden mas 1} satisfies $\risk(p) < \riskas$ for any $p \in (0,1)$. Similarly, for the loss function \eqref{eq: L conf gen} with $A_1 > 0$, $A_2=0$, if $\funo \geq \Omega/(\nnum-\sqrt\nnum)$ the inequality $\risk(p) < \riskas$ holds for any $p \in (0,1)$.
\end{lemma}

\begin{proof}
The stated results follow from the arguments used in the proof of \citet[proposition~3]{Mendo10}.
\end{proof}

\begin{proof}[Proof of Proposition~\ref{prop: conf gen2 gar}]
The result will be proved by approximating the loss function as a sum of terms of the form \eqref{eq: L conf gen} with $A_1, A_2 \geq 0$ and using Lemma~\ref{lemma: conf gen gar}. It may be assumed without loss of generality that $L(x)=0$ for $x \in [\upsilon, \upsilon']$, because if $L(x)=C$ within that interval, defining $L'(x)=L(x)-C$ the risk corresponding to $L$ is expressed as $C$ plus the risk resulting from the loss function $L'$, which satisfies the hypotheses of the proposition.

Let $\epsilon>0$, and suppose for the moment that $L$ is unbounded on the interval $(0,\upsilon)$. This implies that for any $i \in \mathbb N$, the set $V_{\epsilon,i}=\{x \in (0,\upsilon) \st L(x) \geq i\epsilon\}$ is non-empty. In fact, since $L$ is non-increasing on $(0,\upsilon)$, $V_{\epsilon,i}$ is an interval. Let $x_{\epsilon,i}$ be defined as the supremum of $V_{\epsilon,i}$, and let
\begin{equation}
\ell_{\epsilon,i}(x) =
\begin{cases}
\epsilon & \text{if } x \leq x_{\epsilon,i}, \\
0 & \text{otherwise}.
\end{cases}
\end{equation}
If $L$ is bounded on $(0,\upsilon)$, the sets $V_{\epsilon,i}$ are empty for $i$ greater than a certain value. In this case, the corresponding $\ell_{\epsilon,i}$ functions are defined as the null function. In a similar manner, for $L$ unbounded on $(\upsilon',\infty)$, let $V'_{\epsilon,i}=\{x \in (\upsilon',\infty) \st L(x) \geq i\epsilon\}$, which is again non-empty interval; let $x'_{\epsilon,i}$ be its infimum, and
\begin{equation}
\label{eq: L' epsilon i}
\ell'_{\epsilon,i}(x) =
\begin{cases}
\epsilon & \text{if } x \geq x'_{\epsilon,i}, \\
0 & \text{otherwise}.
\end{cases}
\end{equation}
If $L$ is bounded on $(\upsilon',\infty)$, for $i$ greater than a certain value the sets $V'_{\epsilon,i}$ are empty, and the corresponding $\ell'_{\epsilon,i}$ are defined as null. Let $L_{\epsilon,i}(x) = \ell_{\epsilon,i}(x) + \ell'_{\epsilon,i}(x)$ and $L_\epsilon(x) = \sum_{i=1}^\infty L_{\epsilon,i}(x)$. By construction, for all $x \in \mathbb R^+$,
\begin{equation}
\label{eq: L aprox epsilon}
0 \leq L(x)-L_\epsilon(x) \leq \epsilon.
\end{equation}

Each function $L_{\epsilon,i}$ satisfies Assumptions~\ref{assum: L var acot}--\ref{assum: L cota O},
and therefore a risk can be defined considering $L_{\epsilon,i}$ as the loss function. This risk will be denoted as $\risk_{\epsilon,i}(p)$. The function $L_\epsilon$ also satisfies Assumptions~\ref{assum: L var acot}--\ref{assum: L cota O}.
Let $\risk_\epsilon(p)$ denote its corresponding risk,
\begin{equation}
\label{eq: risk epsilon 1}
\risk_\epsilon(p) = \sum_{\nden=\nnum}^\infty f(\nden) L_\epsilon(\fest(\nden)/p)
= \sum_{\nden=\nnum}^\infty \sum_{i=1}^\infty f(\nden) L_{\epsilon,i}(\fest(\nden)/p)
\end{equation}
For each $\nden$, the inner series in \eqref{eq: risk epsilon 1} converges absolutely; namely, to $f(\nden) L_\epsilon(\fest(\nden)/p)$. In addition, from \eqref{eq: L aprox epsilon} it is seen that $L_\epsilon(\fest(\nden)/p) \leq L(\fest(\nden)/p)$, and this implies that the outer series in \eqref{eq: risk epsilon 1} is also absolutely convergent. This allows interchanging the sums over $\nden$ and $i$ \citep[theorem 8.43]{Apostol74}, which gives
\begin{equation}
\label{eq: risk epsilon 2}
\risk_\epsilon(p) = \sum_{i=1}^\infty \risk_{\epsilon,i}(p).
\end{equation}

Theorem~\ref{teo: Feln incl Felp} assures that $\risk_{\epsilon,i}(p)$ has an asymptotic value $\riskas_{\epsilon,i}$, given by
\begin{equation}
\riskas_{\epsilon,i} = \int_0^\infty \phi(\nu) L_{\epsilon,i}(\Omega/\nu) \,\diff\nu, \\
\end{equation}
Similarly, $\risk_\epsilon(p)$ has an asymptotic value
\begin{equation}
\label{eq: riskas epsilon}
\riskas_\epsilon = \int_0^\infty \phi(\nu) \sum_{i=1}^\infty L_{\epsilon,i}(\Omega/\nu) \,\diff\nu.
\end{equation}
Since $L_{\epsilon,i}$ is a nonnegative function for all $i$, the monotone convergence theorem \citep[theorem 2.3.4]{Athreya06} implies that the sum and integral signs in \eqref{eq: riskas epsilon} commute, and thus
\begin{equation}
\label{eq: riskas epsilon 1}
\riskas_\epsilon = \sum_{i=1}^\infty \riskas_{\epsilon,i}.
\end{equation}
From Lemma~\ref{lemma: conf gen gar}, $\risk_{\epsilon,i}(p) < \riskas_{\epsilon,i}$. Combined with \eqref{eq: risk epsilon 2} and \eqref{eq: riskas epsilon 1}, this gives
\begin{equation}
\label{eq: risk epsilon riskas epsilon}
\risk_\epsilon(p) < \riskas_\epsilon.
\end{equation}
On the other hand, from \eqref{eq: L aprox epsilon} it stems that
\begin{equation}
\label{eq: risk risk epsilon}
0 \leq \risk(p)-\risk_\epsilon(p) \leq \epsilon,
\end{equation}
which in turn implies
\begin{equation}
\label{eq: riskas riskas epsilon}
0 \leq \riskas-\riskas_\epsilon \leq \epsilon.
\end{equation}
From \eqref{eq: risk epsilon riskas epsilon}--\eqref{eq: riskas riskas epsilon},
\begin{equation}
\label{eq: risk riskas}
\risk(p) \leq \risk_\epsilon(p) + \epsilon < \riskas_\epsilon + \epsilon  < \riskas + \epsilon.
\end{equation}
Since \eqref{eq: risk riskas} holds for $\epsilon$ arbitrary, the desired inequality $\risk(p) \leq \riskas$ follows.
\end{proof}


\end{document}